\documentclass[12pt,reqno]{amsart}
\usepackage[toc,page]{appendix}
\newcommand{\stoptocwriting}{%
  \addtocontents{toc}{\protect\setcounter{tocdepth}{-5}}}
\newcommand{\resumetocwriting}{%
  \addtocontents{toc}{\protect\setcounter{tocdepth}{\arabic{tocdepth}}}}

\usepackage{amsmath}
\usepackage{xcolor}
\usepackage{placeins}
\usepackage{amsfonts,amsmath,amsthm}
\usepackage{amssymb,epsfig}
\usepackage{enumerate}
\usepackage{fullpage}
\usepackage[utf8]{inputenc}
\usepackage{mathpazo}
\usepackage{mathtools}

\usepackage{eucal}
\usepackage[euler-digits]{eulervm}
\usepackage[unicode=true]{hyperref}
\hypersetup{colorlinks = true}
\hypersetup{
     colorlinks,
     linkcolor={black!10!red},
     linkbordercolor = {black!100!red},
     citecolor={blue}
}

\usepackage{pdfsync}

\usepackage{geometry}
\usepackage{graphicx}
\usepackage{epsfig}
\usepackage{tikz}
\usepackage{caption}
\usepackage{color} %color
\definecolor{vert}{rgb}{0,0.6,0}

\usepackage{comment}
\numberwithin{figure}{section}

\theoremstyle{plain}
\newtheorem{thm}{Theorem}[section]

\newtheorem{defn}{Definition}

\newtheorem{lem}[thm]{Lemma}
\newtheorem{cor}[thm]{Corollary}

\theoremstyle{remark}
\newtheorem{rem}{\bf{Remark}}
\numberwithin{equation}{section}

\newcommand{\R}{\mathbb{R}}

\newcommand{\rmC}{\mathrm{C}}
\newcommand{\rup}{\rightharpoonup}

\usepackage{import}
\usepackage{xifthen}
\usepackage{pdfpages}
\usepackage{transparent}

\begin{document}
\title[Rate of convergence]
{\textsc{Remarks on the vanishing viscosity process of state-constraint Hamilton--Jacobi equations}}
\thanks{The authors are supported in part by NSF grant DMS-1664424 and NSF CAREER grant DMS-1843320. The work of Son N. T. Tu is supported in part by the GSSC Fellowship, University of Wisconsin--Madison.}
\begin{abstract}
We investigate the convergence rate in the vanishing viscosity process of the solutions to the subquadratic state-constraint Hamilton-Jacobi equations. We give two different proofs of the fact that, for nonnegative Lipschitz data that vanish on the boundary, the rate of convergence is $\mathcal{O}(\sqrt{\varepsilon})$ in the interior. Moreover, the one-sided rate can be improved to $\mathcal{O}(\varepsilon)$ for nonnegative compactly supported data and $\mathcal{O}(\varepsilon^{1/(p-\frac{1}{2})})$ (where $1<p\leq 2$ is the exponent of the gradient term) for nonnegative data $f\in \mathrm{C}^2(\overline{\Omega})$ such that $f = 0$ and $Df = 0$ on the boundary. Our approach relies on deep understanding of the blow-up behavior near the boundary and semiconcavity of the solutions.
\end{abstract}
%%%%%%%%%%%%%%%%%%%%%%%%%%%%%%%%%%%%%%%%%%%%%%%%%%%%%%%%%%%
\author{Yuxi Han}
\address[Y. Han]
{
Department of Mathematics,
University of Wisconsin Madison, 480 Lincoln  Drive, Madison, WI 53706, USA}
\email{yuxi.han@wisc.edu}
\author{Son N. T. Tu}
\address[S. N.T. Tu]
{
Department of Mathematics,
University of Wisconsin Madison, 480 Lincoln  Drive, Madison, WI 53706, USA}
\email{thaison@math.wisc.edu}
\date{July 31, 2025.}
\maketitle
\setcounter{tocdepth}{1}
%\tableofcontents

%%%%%%%%%%%%%%%%%%%%%%%%%%%%%%%%%%%%%%%%%%%%%%%%%%%%%%%%%%%
\section{Introduction}\label{sec:intro}
\subsection{Settings} Let $\Omega$ be an open, bounded and connected domain in $\mathbb{R}^n$ with $\mathrm{C}^2$ boundary, $f\in \mathrm{C}(\overline{\Omega})\cap W^{1,\infty}(\Omega)$. For $\varepsilon>0$, let $u^\varepsilon\in \mathrm{C}^2(\Omega)$ (see \cite{Lasry1989}) be the solution to
 \begin{equation}\label{eq:PDEepsa}
    \begin{cases}
      u^\varepsilon(x) + H(Du^\varepsilon(x)) - f(x) - \varepsilon \Delta u^\varepsilon(x) = 0 \qquad
    \text{in}\;\Omega, \vspace{0cm}\\
    \displaystyle  \lim_{\mathrm{dist}(x,\partial \Omega)\to 0} u^\varepsilon(x) = +\infty
    \end{cases}
\end{equation}
 where $H:\R^n\to\R^n$ is a given continuous Hamiltonian. The solution that blows up uniformly on the boundary is also called a \emph{large solution}. A typical subquadratic Hamiltonian that has been considered in the literature is $H(\xi) = |\xi|^p$ for $\xi\in \R^n$ where $1<p\leq 2$. We focus on this Hamiltonian in our paper, and the equation of interest becomes
\begin{equation}\label{eq:PDEeps}
    \begin{cases}
      u^\varepsilon(x) + |Du^\varepsilon(x)|^p - f(x) - \varepsilon \Delta u^\varepsilon(x) = 0 \qquad
    \text{in}\;\Omega, \vspace{0cm}\\
    \displaystyle  \lim_{\mathrm{dist}(x,\partial \Omega)\to 0} u^\varepsilon(x) = +\infty.
    \end{cases} \tag{PDE$_\varepsilon$}
\end{equation}
When $1<p\leq 2$, equation \eqref{eq:PDEeps} describes the value function associated with a minimization problem in stochastic optimal control with state constraints (\cite{fabbri_stochastic_2017, Lasry1989}). We briefly recall the setting as follows. For a given stochastic control $\alpha(\cdot)$, we look for a solution (a state process) of the feedback control system
\begin{equation}\label{eq:u2}
\begin{cases}
dX_t = \alpha\left(X_t\right)dt + \varepsilon\sqrt{2}\,d\mathbb{B}_t \qquad \text { for } t >0,\\
\;\; X_0 = x.
\end{cases}
\end{equation}
We omit all the domains and destination spaces for simplicity. Here, $\mathbb{B}_t\sim \mathcal{N}(0,t)$ is the Brownian motion with mean zero and variance $t$. To constrain the state $X_t$ inside $\overline{\Omega}$, we define
\begin{equation*}
    \widehat{\mathcal{A}}_x = \Big\lbrace\alpha(\cdot)\in \mathrm{C}(\Omega): \mathbb{P}(X_t\in \Omega) = 1\;\text{for all}\;t\geq 0\Big\rbrace
\end{equation*}
and hope to minimize the cost function
\begin{equation*}
    u^\varepsilon(x) = \inf_{\alpha\in \widehat{\mathcal{A}}_x} \mathbb{E}\left[\int_0^\infty e^{-t}L\big(X_t,\alpha(X_t) \big)\;dt\right],
\end{equation*}
where $L(x,v):\overline{\Omega}\times \mathbb{R}^n \to \mathbb{R}$ is the running cost. In our setting, we consider the Lagrangian $L(x,v) = c|v|^q+f(x)$ from the classical mechanics, where $q>1$, $f\in \mathrm{C}(\overline{\Omega})$ is nonnegative, and $c$ is chosen so that the corresponding Hamiltonian (defined via the Legendre transform) is $H(x,\xi) = |\xi|^p - f(x)$. Using the Dynamic Programming Principle (see
\cite{Lasry1989}), we expect the value function \eqref{eq:u2} to satisfy the Hamilton--Jacobi equation with state constraint, that is,
\begin{equation}\label{eq:HJB}
\begin{cases}
    u^\varepsilon(x) + |Du^\varepsilon(x)|^p - f(x) - \varepsilon \Delta u^\varepsilon(x) \leq  0 \qquad\text{in}\;\Omega,\\
   u^\varepsilon(x) + |Du^\varepsilon(x)|^p - f(x) - \varepsilon \Delta u^\varepsilon(x) \geq  0 \qquad\text{on}\;\overline{\Omega},
\end{cases}
\end{equation}
in the viscosity solution framework. It means that $u^\varepsilon$ is a subsolution in $\Omega$ and a supersolution on $\overline{\Omega}$. It is clear that any solution of \eqref{eq:PDEeps} satisfies \eqref{eq:HJB}. It turns out that when $1<p
\leq 2$, the boundary condition \eqref{eq:PDEeps} correctly describes the state-constraint problem \eqref{eq:HJB} (see \cite{Lasry1989}).

We are interested in studying the asymptotic behavior of $\{u^\varepsilon\}_{\varepsilon>0}$ as $\varepsilon\rightarrow 0^+$. Heuristically, the solution of the second-order state-constraint problem converges to that of a first-order state-constraint problem associated with the deterministic optimal control, namely,
\begin{equation}\label{eq:PDE0}
    \begin{cases}
       u(x) + |Du(x)|^p - f(x) \leq 0\;\qquad\text{in}\;\Omega,\\
       u(x) + |Du(x)|^p - f(x) \geq 0\;\qquad\text{on}\;\overline{\Omega}.
    \end{cases} \tag{PDE$_0$}
\end{equation}
 Equation \eqref{eq:PDE0} admits a unique viscosity solution in the space $\rmC(\overline{\Omega})$, which is also the maximal viscosity subsolution among all the viscosity subsolutions in $\rmC(\overline{\Omega})$(see \cite{Capuzzo-Dolcetta1990,Soner1986}). In terms of optimal control, as $\varepsilon \to 0^+$, the stochastic control system \eqref{eq:u2} becomes a deterministic control system. In particular, let $\mathcal{A}_x = \{\zeta\in \mathrm{AC}([0,\infty);\overline{\Omega}): \zeta(0)=x\}$ and we have
\begin{equation*}
    u(x) = \inf_{\zeta\in \mathcal{A}_x} \int_0^\infty e^{-t}L\big(\zeta(t),\do{\zeta}(t)\big)dt
\end{equation*}
where $L(x,v)$ is the Legendre transform of $H(x,\xi) = |\xi|^p - f(x)$.

The problem is interesting since in the limit we no longer have blowing up behavior, as $u\in \mathrm{C}(\overline{\Omega})$. In this paper, we investigate the rate of convergence of $u^\varepsilon \to u$ as $\varepsilon\to 0^+$. What is intriguing and delicate here is the blow-up behavior of $u^\varepsilon$ in a narrow strip near $\partial \Omega$ as $\varepsilon\to 0^+$. This is often called the boundary layer theory in the literature.

Note that a comparison principle holds for \eqref{eq:PDE0} since we always assume $\Omega$ is an open, bounded and connected domain in $\mathbb{R}^n$ with $\mathrm{C}^2$ boundary (\cite{Capuzzo-Dolcetta1990,Soner1986}). Equation \eqref{eq:PDEeps} follows the setting of \cite{Lasry1989}, where the specific structure of the Hamiltonian $H(x,\xi) = |\xi|^p - f(x)$ enables more explicit estimates for the solution of \eqref{eq:PDEeps}.

\subsection{Relevant literature} There is a
vast amount of work in the literature on viscosity solutions with state constraints and large solutions. We would like to first mention that the problem \eqref{eq:PDE0} with general Hamiltonian is a huge subject of research interest, started with the pioneer work \cite{Soner1986} (see also \cite{ishii_new_1996,ishii_class_2002}).  Some of the recent work related to the asymptotic behavior of solutions of \eqref{eq:PDE0} can be found in \cite{ishii_vanishing_2017,kim_state-constraint_2020,mitake_asymptotic_2008,tu2021vanishing}. The problem \eqref{eq:PDEeps}
was first studied in \cite{Lasry1989} and subsequently many works have been done in understanding deeper the properties of solutions (see \cite{ Porretta_a,alessio_asymptotic_2006,Bandle_1994, marcus_existence_2003} and the references therein). The time-dependent version of \eqref{eq:PDEepsa} was also studied by many works, for instance, \cite{barles_generalized_2004,barles_large_2010,leonori_local_2011,moll_large_2012} and the references therein. %For deterministic and stochastic optimal control problems with state constraints, we refer the read to [\textcolor{blue}{NEED CITATION (Ishii, \cite{}}] and the references therein.

In terms of rate of convergence, that is, the convergence rate of $u^\varepsilon\to u$ as $\varepsilon\to 0^+$, to the best of our knowledge, such a question has not been studied in the literature. For the case where \eqref{eq:PDEeps} is equipped with the Dirichlet boundary condition, a rate $\mathcal{O}(\sqrt{\varepsilon})$ is well known with multiple proofs (see \cite{Bardi1997,crandall1984, tran_hamilton-jacobi_2021}). We believe part of the reason why the problem is intriguing is that the blow-up behavior of the solutions near the boundary of $\Omega$ is complicated and deserves further investigation.

\subsection{Main results} Define
\begin{equation*}
    \alpha = \frac{2-p}{p-1} \in [0,\infty).
\end{equation*}
Without loss of generality, we can assume $\min_{\overline{\Omega}} f = 0$. The main results of the paper are the following theorems.

\begin{thm}\label{main_thm1} Let $\Omega$ be an open, bounded and connected subset of $\R^n$ with $\mathrm{C}^2$ boundary. Assume that $1 < p\leq 2$ and $f$ is nonnegative and Lipschitz with $f = 0$ on $\partial\Omega$. Let $u^\varepsilon$ be the unique solution to \eqref{eq:PDEeps} and $u$ be the unique solution to \eqref{eq:PDE0}. Then there exists a constant $C$ independent of $\varepsilon\in (0,1)$ such that for $x\in \Omega$,
\begin{align*}
    &-C\sqrt{\varepsilon}\leq  u^\varepsilon(x) - u(x)\leq C\left(\sqrt{\varepsilon} + \frac{\varepsilon^{\alpha+1}}{d(x)^\alpha}\right), \qquad\; 1<p < 2,\\
    &-C\sqrt{\varepsilon}\leq  u^\varepsilon(x) - u(x)\leq C\left(\sqrt{\varepsilon} + \varepsilon|\log(d(x))|\right), \qquad \; p = 2.
\end{align*}
\end{thm}

\begin{rem} To the best of our knowledge, this theorem is new in the literature. The precise boundary behavior is very delicate and deserves further investigation. The condition $f = 0$ on $\partial\Omega$ is a little bit restrictive but can be explained as follows. One can see that the solution to \eqref{eq:PDEeps} is continuous with respect to data $f$ in the weak$^*$ topology of $L^\infty(\Omega)$ (see \cite{Lasry1989}). If $f = 0$ on $\partial\Omega$, we
can approximate $f$ \emph{uniformly} in $L^\infty(\Omega)$ by a sequence of compactly supported functions, where a convergence rate of $u^\varepsilon-u$ is easier to obtain. Indeed, for $1<p<2$, it is natural to consider the doubling variable method with
\begin{equation}\label{heur1}
    \psi^\varepsilon(x) :=  u^\varepsilon(x) - \frac{C_\alpha\varepsilon^{\alpha+1}}{d(x)^{\alpha}}
\end{equation}
and $u(x)$, where $d(x)$ is a $\mathrm{C}^2$ extension of the distance function to the boundary and $C_\alpha\varepsilon^{\alpha+1}d(x)^{-\alpha}$ is the leading order term in the asymptotic expansion of $u^\varepsilon(x)$ as $d(x) \to 0^{+}$ with $C_\alpha:= \alpha^{-1}(\alpha+1)^{\alpha+1}$. If we take the derivative of \eqref{heur1} formally, it becomes
\begin{equation}\label{heur2}
    D\psi^\varepsilon(x) =  Du^\varepsilon(x) + C_\alpha\alpha \left(\frac{\varepsilon}{d(x)}\right)^{\alpha+1}Dd(x).
\end{equation}
We will see that $D\psi^\varepsilon(x)$ is uniformly bounded if $d(x)\geq \varepsilon$ (Lemma \ref{lem:boundDu^eps}). Indeed,
\begin{equation*}
    -C_\alpha \alpha \left(\frac{\varepsilon}{d(x)}\right)^{\alpha+1}Dd(x)
\end{equation*}
is more or less the leading order term in the asymptotic expansion of $Du^\varepsilon$ near $\partial\Omega$. Heuristically, this means that the boundary layer is $\mathcal{O}(\varepsilon)$ from the boundary.

However, to get a useful estimate by the doubling variable method, at the maximum point $x_0$ of $\psi^\varepsilon(x) - u(x)$, we need to have $d(x_0)\geq \varepsilon^{\gamma}$ for $\gamma<1$ so that the latter term in \eqref{heur2} vanishes as $\varepsilon\to 0^+$. In the other case where $d(x_0) < \varepsilon^{\gamma}$, we introduce a new localization idea, that is, we construct a blow-up solution in the ball of radius $\varepsilon^\gamma$ from the boundary. Finally, a technical (and common for the doubling variable method) computation leads to $\gamma = 1/2$.

%Essentially we avoid the boundary layer here.

%In the complicated case that $x_0$ is in the boundary layer, we introduce a new localization idea, that is, we construct a blow-up solution in the ball of radius $\varepsilon$.
%\textcolor{red}{(Actually, I am thinking should it be the ball of radius $\sqrt{\varepsilon}$ from the boundary? Since we choose $\kappa=\sqrt{\varepsilon}$.)}
\end{rem}

As a different approach, the convexity of $|\xi|^p$ and the semiconcavity of the solution to \eqref{eq:PDE0} give us a better one-sided $\mathcal{O}(\varepsilon)$ estimate for nonnegative compactly supported $f$ which is semiconcave in its support (see Theorem \ref{thm:rate_doubling2}). Such an one-sided $\mathcal{O}(\varepsilon)$ rate is well known for the Dirichlet boundary problem (see \cite{Bardi1997,tran_hamilton-jacobi_2021}). Moreover, the result in Theorem \ref{thm:rate_doubling2} further provides us with a better one-sided estimate $\mathcal{O}(\varepsilon^{1/p})$ than that in Theorem \ref{main_thm1}, as in Corollary \ref{cor:key}. We recall that $f$ is (uniformly) semiconcave in $\overline{\Omega}$ with linear modulus (or semiconcavity constant) $c>0$ if
\begin{equation*}
    f(x+h)-2f(x)+f(x-h)\leq c|h|^2, \quad \forall x, h \in \mathbb{R}^n \,\, \text{such that} \, \, x+h, x, \text{and}\, \, x-h \in \overline{\Omega}.
\end{equation*}
Note that any $f\in \mathrm{C}^2(K)$ with $K$ compact is semiconcave on $K$ with the constant
$$c = \max \left\{ D_{\xi\xi}f(x): |\xi|=1, x\in K\right\}$$
in the above definition.
% this approach provides us with a better one-sided estimate $\mathcal{O}(\varepsilon^{1/p})$ than that in Theorem \ref{main_thm1} for nonnegative Lipschitz data that vanish on the boundary, without using the doubling variable method.
It is well known that the solution $u$ to \eqref{eq:PDE0} is \emph{locally} semiconcave given $f$ is uniformly semiconcave in $\overline{\Omega}$. Using tools from the optimal control theory, we provide the explicit blow-up rate of the semiconcavity modulus of $u(x)$ when $x$ approaches $\partial\Omega$. As an application, we can improve the rate of convergence as follows.

\begin{thm}[One-sided $\mathcal{O}(\varepsilon)$ rate for nonnegative compactly supported data]\label{thm:rate_doubling2}

% Define $   \delta_{0} :=\frac{1}{2}\sup \big\{ \delta > 0: x\mapsto\mathrm{dist}(x,\partial\Omega)\;\text{is}\;\rmC^2\;\text{in}\;\Omega^\delta \backslash\overline{\Omega}_{\delta} \big\}$.

% Assume that $f\in \mathrm{Lip}(\Omega)$ is nonnegative with compact support in $\Omega_\kappa := \{x\in \Omega: \mathrm{dist}(x,\Omega) > \kappa\}$ for $\kappa$ small enough, and semiconcave in its support. Let $u^\varepsilon$ and $u$ be the unique solutions to \eqref{eq:PDEeps} and \eqref{eq:PDE0} respectively. Then there exist two constants $\nu > 1$ and $C$ independent of $\varepsilon$  and $\kappa$ such that
% Suppose that $f$ is a non-negative and Lipschitz function, $u^\varepsilon$ and $u$ are the unique solutions to \eqref{eq:PDEeps} and \eqref{eq:PDE0} respectively. There exists a subset $\Omega_k\subset\Omega$ such that if $f$ is semiconcave and supported in $\Omega_k$,

Under the conditions of Theorem \ref{main_thm1}, suppose $f$ also satisfies the following conditions:
\begin{itemize}
    \item $f$ is semiconcave in its support;
    \item $f$ has a compact support in $\Omega_\kappa := \{x\in \Omega: \mathrm{dist}(x,\Omega) > \kappa\}$ for some $\kappa\in(0,\delta_0)$\,, where $\delta_0$ is defined in \eqref{def:delta_0}.
\end{itemize}
Then there exist two constants $\nu > 1$ and $C$ independent of $\varepsilon$  and $\kappa$ such that
\begin{equation*}
\begin{aligned}
    u^\varepsilon(x) - u(x) &\leq\frac{\nu C_\alpha \varepsilon^{\alpha+1}}{d(x)^\alpha} + C \left(\left(\frac{\varepsilon}{\kappa}\right)^{\alpha+1}+\left(\frac{\varepsilon}{\kappa}\right)^{\alpha+2}\right) + \frac{Cn\varepsilon}{\kappa},  &\quad p <2, \\
    u^\varepsilon(x) - u(x) & \leq \nu \varepsilon \log\left( \frac{1}{d(x)}\right)+C \left(\left(\frac{\varepsilon}{\kappa}\right)+\left(\frac{\varepsilon}{\kappa}\right)^2\right)+ \frac{Cn\varepsilon}{\kappa} , &\quad p=2.
\end{aligned}
\end{equation*}
\end{thm}

\begin{rem}\label{rem:C} If $f\in \mathrm{C}^2_c(\R^n)$, then the last term $\left(Cn\varepsilon\right)\kappa^{-1}$ in the equations above can be improved to $nc\varepsilon$, where $c$ is the semiconcavity constant of $f$.
\end{rem}

\begin{figure}[ht]
    \centering
    %\incfig{Domains}
    \includegraphics[scale = 0.35]{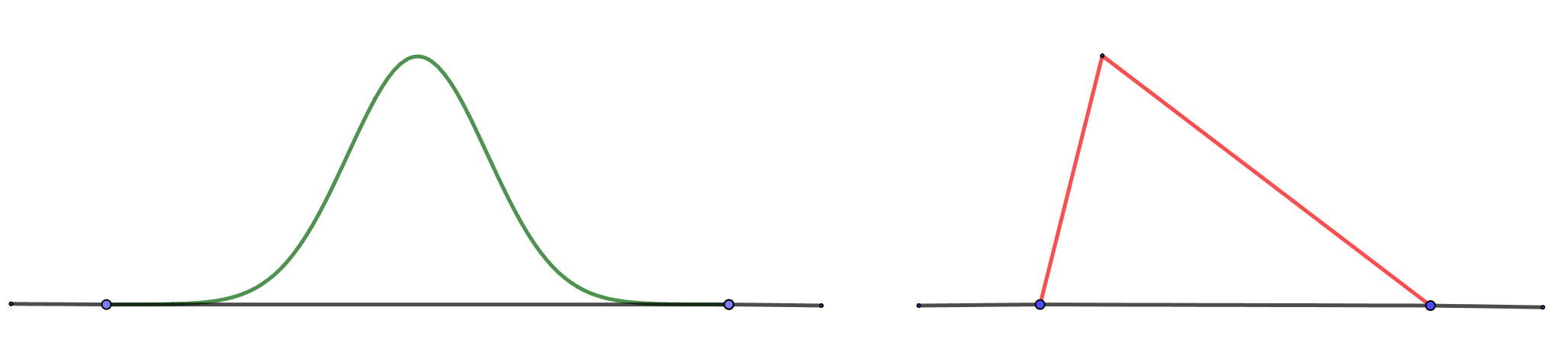}
    \caption{Two different cases where $f$ can and cannot be extended to a semiconcave function in the whole space by setting $f = 0$ outside $\Omega$. The one on the right corresponds to the assumption in Theorem
    \ref{thm:rate_doubling2}, while the one on the left corresponds to the assumption in Remark \ref{rem:C}.}
    \label{fig:Domains2}
\end{figure}

\begin{cor}[One-sided $\mathcal{O}(\varepsilon^{1/(p-\frac{1}{2})})$ rate]\label{cor:key} Let $1<p\leq 2$. If $f \in \mathrm{C}^2(\overline{\Omega})$ is nonnegative, $f = 0$ and $Df = 0$ on $\partial\Omega$, then there exists a constant $C$ independent of $\varepsilon\in (0,1)$ such that
\begin{align*}
    &-C\varepsilon^{1/2}\leq  u^\varepsilon(x) - u(x)\leq C\left(\varepsilon^{\frac{1}{p-\frac{1}{2}}}+ \frac{\varepsilon^{\alpha+1}}{d(x)^\alpha}\right)
\end{align*}
for all $x\in \Omega$.
\end{cor}

\begin{rem} While the second approach looks more powerful, we need the gradient bound of $u^\varepsilon$ (Lemma \ref{lem:boundDu^eps}), the blow-up rate of the semiconcavity constant of $u$ (Theorem \ref{thm:newsemi}), and higher regularity on $f$. On the other hand, the first approach by doubling variable is relatively simple and does not require any explicit asymptotic behavior of $Du^\varepsilon$, except the fact that it is locally bounded. %At this moment, we do not know how to generalize Theorem \ref{thm:rate_doubling2} to the case where $f\in \mathrm{Lip}(\Omega)$ is nonnegative and semiconcave with $f = 0$ on $\partial \Omega$.
\end{rem}

\begin{rem}
In all of the theorems above, $f$ is assumed to be nonnegative with $f=0$ on $\partial \Omega$. Note that by adding a constant to the solution $u^\varepsilon$ of \eqref{eq:PDEeps}, we can see that all the results hold true for Lipschitz data that equals to its minimum on the boundary, meaning
\[
f(z)=\min_{x\in\partial\Omega}f(x)=\min_{x\in\overline{\Omega}}f(x),\quad \forall z\in\partial\Omega.
\]
%\textcolor{blue}{Also, if $f\in \mathsf{C}^2(\overline{\Omega})$ with $\Omega$ is compact then $f$ is semiconcave on $\overline{\Omega}$.}
\end{rem}

\subsection*{Organization of the paper} Section 2 contains some preliminary results. The proof of Theorem \ref{main_thm1} is given in Section 3. Then in Section 4, we give the proof of Theorem \ref{thm:rate_doubling2} and Corollary \ref{cor:key}. Finally, the proofs of some useful lemmas are presented in Appendix.

\section{Preliminaries}\label{sec:prelim}
Let $\Omega$ be an open, bounded and connected subset of $\mathbb{R}^n$ with boundary $\partial\Omega$ of class $C^2$. For small $\delta>0$, we denote $\Omega_\delta = \{x\in \Omega: \mathrm{dist}(x,\Omega) > \delta\}$ and $\Omega^\delta = \{x\in \mathbb{R}^n: \mathrm{dist}(x,\overline{\Omega}) < \delta\}$.
\begin{figure}[ht]
    \centering
    %\incfig{Domains}
    \includegraphics[scale = 0.45]{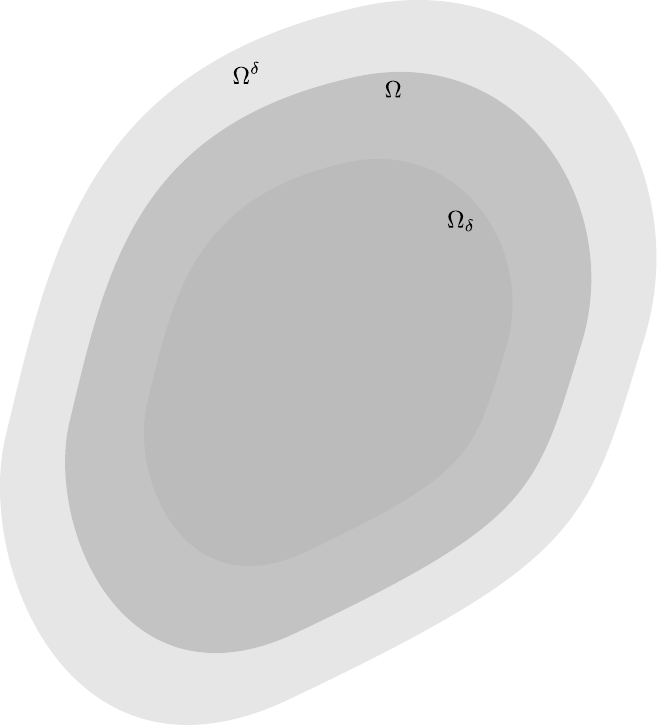}
    \caption{The domain $\Omega$ and its variations $\Omega_\delta, \Omega^\delta$.}
    \label{fig:Domains}
\end{figure}

\begin{defn} Define
\begin{equation}\label{def:delta_0}
    \delta_{0,\Omega} =\frac{1}{2}\sup \big\{ \delta > 0: x\mapsto\mathrm{dist}(x,\partial\Omega)\;\text{is}\;\rmC^2\;\text{in}\;\Omega^\delta \backslash\overline{\Omega}_{\delta} \big\}.
\end{equation}
%The distance function is of class $\mathrm{C}^2$ in the region where $0<\mathrm{dist}(x,\partial\Omega) < \delta_0$.
We will simply write $\delta_0$ instead of $\delta_{0,\Omega}$ when the underlying domain is understood.
\end{defn}

We refer the reader to \cite{gilbarg_elliptic_2001} for the regularity of the distance function defined in the definition above. We then extend $\mathrm{dist}(x,\partial\Omega)$ to a function $d(x)\in \mathrm{C}^2(\mathbb{R}^n)$ such that
\begin{equation}\label{e:distance_def}
    \begin{cases}
    d(x)\geq 0\;\text{for}\;x\in\Omega\;\text{with}\;d(x) = +\mathrm{dist}(x,\partial\Omega)\;\text{for}\;x\in \Omega\backslash \Omega_{\delta_0},\\
    d(x)\leq 0\;\text{for}\;x\notin \Omega\;\text{with}\;d(x) = -\mathrm{dist}(x,\partial\Omega)\;\text{for}\;x\in \Omega^{\delta_0}\backslash \Omega.
    \end{cases}
\end{equation}
Note that $d(x) = \mathrm{dist}(x,\partial\Omega)$ and $|D d(x)| = 1$ in the classical sense in $\Omega^{\delta_0}\backslash \Omega_{\delta_0}$. Let
\begin{equation}\label{boundond}
   K_0:= \max_{x\in \overline{\Omega}}|d(x)|, \qquad K_1 := \max_{x\in \overline{\Omega}} |D d(x)|, \qquad\text{and}\qquad K_2 := \max_{x\in \overline{\Omega}} |\Delta d(x)|.
\end{equation}
\noindent Denote by $\mathcal{L}^\varepsilon:\rmC^2(\Omega)\to \rmC(\Omega)$ the operator
\begin{equation*}
    \mathcal{L}^\varepsilon[u](x) :=   u(x) + |Du(x)|^p - f(x) - \varepsilon \Delta u(x), \qquad x\in \Omega.
\end{equation*}

\subsection{Local gradient estimate} For $\varepsilon \in (0,1)$ and $p>1$, we state an a priori estimate for $\rmC^2$ solutions to \eqref{eq:PDEeps} (\cite[Appendix]{Lasry1989}). Since we are working with nice solutions, the proof is relatively simple by the classical Bernstein method (\cite{bernstein_sur_1910}), which is provided in Appendix for the reader's convenience.

\begin{thm}\label{thm:grad_1} Let $f\in \rmC(\overline{\Omega})\cap W^{1,\infty}(\Omega)$ and $u^\varepsilon \in \mathrm{C}^2(\Omega)$ be a solution to $\mathcal{L}^\varepsilon[u^\varepsilon] = 0$ in $\Omega$. Let $m:= \max_{\overline{\Omega}}f(x)$. Then for $\delta>0$, there exists $C_\delta = C(m,p,\delta, \Vert D f\Vert_{L^\infty(\Omega)})$ such that
\begin{equation*}
    \sup_{x\in \overline{\Omega}_\delta} \Big(|u^\varepsilon(x)|+|Du^\varepsilon(x)|\Big) \leq C_\delta
\end{equation*}
for $\varepsilon$ small enough.
\end{thm}

%\textcolor{orange}{(I think here we want to add for $\varepsilon$ small enough?)}

\subsection{Well-posedness of \eqref{eq:PDEeps}} In this section, we recall the existence and the uniqueness of solutions to \eqref{eq:PDEeps} for $1<p\leq 2$ and $f\in \mathrm{C}(\overline{\Omega})\cap W^{1,\infty}(\Omega)$. In fact, the assumption of $f$ can be relaxed to $f\in L^\infty(\Omega)$ (\cite{Lasry1989}).

\begin{thm}\label{thm:wellposed1<p<2} Let $f\in \mathrm{C}(\overline{\Omega})\cap W^{1,\infty}(\Omega)$. There exists a unique solution $u^\varepsilon\in \mathrm{C}^2(\Omega)$ of \eqref{eq:PDEeps} such that:
\begin{itemize}
    \item[(i)] If $1<p< 2$, then
\begin{equation}\label{rate_p<2}
    %\frac{C_\varepsilon-\eta}{d(x)^\alpha}-\frac{M_\eta}{ } \leq u(x)\leq \frac{C_\varepsilon+\eta}{d(x)^\alpha}+\frac{M_\eta}{ }
    \lim_{d(x)\to 0}\left( u^\varepsilon(x) \,d(x)^\alpha \right)= C_\alpha \varepsilon^{\alpha+1},
\end{equation}
where $\alpha = (p-1)^{-1}(2-p)$ and $C_\alpha = \alpha^{-1}(\alpha+1)^{\alpha+1}$.
%\begin{equation*}
%    \displaystyle\alpha = \frac{2-p}{p-1} \qquad\text{and}\qquad C_\alpha = \frac{1}{\alpha}(\alpha+1)^{\alpha+1}.
%\end{equation*}
\item[(ii)] If $p=2$, then
\begin{equation}\label{rate_p=2}
    \lim_{d(x)\to 0} \left(-\frac{u^\varepsilon(x)}{\log(d(x))}\right) = \varepsilon.
\end{equation}
\end{itemize}
Furthermore, $u^\varepsilon$ is the maximal subsolution among all the subsolutions $v\in W^{2,r}_{\mathrm{loc}}(\Omega)$ for all $r\in [1,\infty)$ of \eqref{eq:PDEeps}.
\end{thm}
\noindent This is Theorem I.1 in \cite{Lasry1989} with an explicit dependence on $\varepsilon$. The proof of this theorem is carried out explicitly in Appendix for later use. Also, it is useful to note that $\alpha+1 = (p-1)^{-1}$. More results on the behavior of the gradient of $u^\varepsilon$ can be found in \cite{alessio_asymptotic_2006} and Lemma \ref{lem:boundDu^eps}, where we show $|Du^\varepsilon|\leq C + C\left(\frac{\varepsilon}{d(x)}\right)^{\alpha+1}$. We believe Lemma \ref{lem:boundDu^eps} is new in the literature.

\subsection{Convergence results} We first state the following lemma (\cite{Capuzzo-Dolcetta1990}), which characterizes the solution to the first-order state-constraint  equation \eqref{eq:PDE0}.
\begin{lem}\label{lem:max} Let $u\in \rmC(\overline{\Omega})$ be a viscosity subsolution of \eqref{eq:PDE0} such that, for any viscosity subsolution $v\in \rmC(\overline{\Omega})$ of \eqref{eq:PDE0}, one has $v\leq u$ on $\overline{\Omega}$. Then $u$ is a viscosity supersolution of \eqref{eq:PDE0} on $\overline{\Omega}$.
\end{lem}
\noindent Again, the proof of Lemma \ref{lem:max} is given in Appendix for the reader's convenience.

\begin{lem}\label{lem:lower-bound} Assume $1<p\leq 2$. Let $u^\varepsilon\in \mathrm{C}^2(\Omega)$ be the solution to \eqref{eq:PDEeps}. We have $\{  u^\varepsilon\}_{\varepsilon>0}$ is uniformly bounded from below by a constant independent of $\varepsilon$. More precisely, $  u^\varepsilon \geq \min_\Omega f$ and $  u\geq \min_\Omega f$.
%\begin{equation}\label{e:lower_bound_u_eps}
%    \inf_{x\in \Omega}   u^\varepsilon(x) \geq  \inf_{\Omega} f
%\end{equation}
%and also
%\begin{equation}\label{e:lower_bound_u}
%    \inf_{x\in \Omega}   u(x) \geq  \inf_{\Omega} f
%\end{equation}
\end{lem}
\begin{proof} For $m\in \mathbb{N}$, let $u^{\varepsilon}_m\in \mathrm{C}^2(\Omega)\cap \rmC(\overline{\Omega})$ solve the Dirichlet problem
\begin{equation}\label{e:uepsm}
    \begin{cases}
      u^{\varepsilon}_m(x) + |Du^{\varepsilon}_m(x)|^p - f(x) - \varepsilon \Delta u^{\varepsilon}_m(x) = 0 &\qquad
    \text{in}\;\Omega, \vspace{0cm}\\
    \quad \qquad\quad\qquad\qquad\qquad\qquad u^{\varepsilon}_m(x) = m &\qquad
    \text{on}\;\partial\Omega.
    \end{cases} \tag{PDE$_{\varepsilon,m}$}
\end{equation}
We have $u^{\varepsilon}_m(x) \to u^\varepsilon(x)$ in $\Omega$ as $m\to \infty$. Let $\varphi(x) \equiv  \inf_{\Omega} f$ for $x\in \overline{\Omega}$. Then $\varphi(x)$ is a classical subsolution of \eqref{e:uepsm} in $\Omega$ with
\begin{equation*}
    \varphi(x) =   \inf_\Omega f \leq m = u^\varepsilon_m(x) \qquad \text{on } \partial \Omega
\end{equation*}
for $m$ large enough.
By the comparison principle of the uniformly elliptic equation \eqref{e:uepsm},
\begin{equation*}
     \inf_\Omega f \leq u^{\varepsilon}_m(x) \qquad\text{for all}\;x\in \Omega.
\end{equation*}
As $m\to \infty$, we obtain $  u^\varepsilon \geq \min_\Omega f$. The inequality $  u\geq \min_{\Omega}f$ follows from the comparison principle of \eqref{eq:PDE0} applied to the supersolution $u$ on $\overline{\Omega}$ and the subsolution $\varphi$ in $\Omega$.
\end{proof}

We present here a simple proof of the convergence $u^\varepsilon \to u$ using Lemma \ref{lem:max}. See also \cite[Theorem VII.3]{Capuzzo-Dolcetta1990}.
\begin{thm}[Vanishing viscosity]\label{thm:qual} Let $u^\varepsilon$ be the solution to \eqref{eq:PDEeps}. Then there exists $u \in \mathrm{C}(\overline{\Omega})$ such that $u^\varepsilon \rightarrow u$ locally uniformly in $\Omega$ as $\varepsilon\rightarrow 0$ and $u$ solves \eqref{eq:PDE0}.
\end{thm}

\begin{proof}[Proof of Theorem \ref{thm:qual}] By the a priori estimate (Theorem \ref{thm:grad_1}),
\begin{equation}\label{e:priorie_eps}
    |u^\varepsilon(x)| + |Du^\varepsilon(x)| \leq C_\delta \qquad\text{for}\;x\in \overline{\Omega}_\delta.
\end{equation}
By the Arzel\`a--Ascoli theorem, there exists a subsequence $\varepsilon_j\to 0$ and a function $u\in \rmC(\Omega)$ such that $u^{\varepsilon_j}\to u$ locally uniformly in $\Omega$. From the stability of viscosity solutions, we easily deduce that
\begin{equation}\label{eq:u0int}
     u(x) + |Du(x)|^p - f(x) = 0 \qquad\text{in}\;\Omega.
\end{equation}
From Lemma \ref{lem:lower-bound}, $  u^\varepsilon(x)\geq \min_{\Omega} f$ and  $  u(x)\geq \min_{\Omega} f$ for all $x\in \Omega$. Together with \eqref{eq:u0int}, we obtain $|\xi|^p \leq \max_\Omega f - \min_\Omega f$ for all $\xi\in D^+u(x)$ and $x\in \Omega$. This implies there exists a constant $C_0$ such that
\begin{equation}\label{e:C0}
    |u(x) - u(y)| \leq C_0|x-y| \qquad\text{for all}\;x,y\in \Omega.
\end{equation}
Thus, we can extend $u$ uniquely to $u\in \rmC(\overline{\Omega})$. We use Lemma \ref{lem:max} to show that $u$ is a supersolution of \eqref{eq:PDE0} on $\overline{\Omega}$.

It suffices to show that $u\geq w$ on $\overline{\Omega}$, where $w\in \rmC(\overline{\Omega})$ is the unique solution to \eqref{eq:PDE0}. For $\delta>0$, let $u_\delta\in\rmC(\overline{\Omega}_\delta)$ be the  unique viscosity solution to
\begin{equation}\label{e:v_v}
    \begin{cases}
      u_\delta(x) + |Du_\delta(x)|^p-f(x) \leq 0 &\qquad\text{in}\;\Omega_\delta,\\
      u_\delta(x) + |Du_\delta(x)|^p - f(x) \geq 0 &\qquad\text{on}\;\overline{\Omega}_\delta.
    \end{cases}
\end{equation}
Since $u_\delta\rightarrow w$ locally uniformly as $\delta\rightarrow 0^+$ (see \cite{kim_state-constraint_2020}) and $w$ is bounded, $\{u_\delta\}_{\delta>0}$ is uniformly bounded. Let $v^\varepsilon_\delta\in \rmC^2(\Omega_\delta)\cap \rmC(\overline{\Omega}_\delta)$ be the unique solution to the Dirichlet problem
\begin{equation}\label{eq:vv_eps}
\begin{cases}
      v_\delta^\varepsilon(x) + |Dv_\delta^\varepsilon(x)|^p - f(x) = \varepsilon \Delta v_\delta^\varepsilon(x) &\qquad\text{in}\;\Omega_\delta,\\
    \;\quad\qquad\qquad\qquad\qquad v_\delta^\varepsilon = u_\delta &\qquad \text{on}\;\partial\Omega_\delta.
\end{cases}
\end{equation}
It is well known that $v^\varepsilon_\delta\to u_\delta$ uniformly on $\overline{\Omega}_\delta$ as $\varepsilon\to 0$.

For $\delta$ small enough, $u_\delta\leq u^\varepsilon$ on $\partial \Omega_\delta$. Hence, by the maximum principle, $v^\varepsilon_\delta \leq u^\varepsilon$ on $\overline{\Omega}_\delta$. Now we first let $\varepsilon\to 0$ to obtain $u_\delta \leq u$ on $\overline{\Omega}_\delta$.
Then let $\delta\rightarrow 0$ to get $w\leq u$ in $\Omega$, which implies $w\leq u$ on $\overline{\Omega}$ since both $w,u$ belong to $\rmC(\overline{\Omega})$.
\end{proof}

\section{Rate of convergence}

In this section, we focus on the rate of convergence for the case where $f\in \mathrm{W}^{1,\infty}(\Omega)\cap\mathrm{C}(\overline{\Omega})$ is nonnegative. As a consequence, $u^\varepsilon(x),u(x)\geq 0$ for $x\in \Omega$ by Lemma \ref{lem:lower-bound}. In our main results, we have an additional assumption that $f=0$ on $\partial \Omega$.

%In this section, we focus on the rate of convergence for the case where $f\in \mathrm{W}^{1,\infty}(\Omega)\cap\mathrm{C}(\overline{\Omega})$ with an extra assumption to be described. By adding a constant to the solution $u^\varepsilon$ of \eqref{eq:PDEeps}, without loss of generality, we can assume the \emph{nonnegativity} of $f$, i.e.,
%\begin{equation}\label{e:minf}
    %\min_{\overline{\Omega}} f(x) = 0.
%\end{equation}
%As a consequence, $u^\varepsilon(x),u(x)\geq 0$ for $x\in \Omega$ by Lemma \ref{lem:lower-bound}.

Before we show any result about the rate of convergence, we would like to mention a lower bound of $u^\varepsilon - u$ and some properties of $u$ from its optimal control formulation.
\begin{thm}\label{thm:lowerbound} Let $u^\varepsilon$ be the unique solution to \eqref{eq:PDEeps} and $u$ be the unique solution to \eqref{eq:PDE0}. Then there exists a constant $C$ independent of $\varepsilon$ such that
\begin{equation}\label{e:lower1}
    -C\sqrt{\varepsilon} \leq u^\varepsilon(x) - u(x) \qquad\text{for all}\;x\in \Omega.
\end{equation}
\end{thm}
\begin{proof} The proof relies on a well-known rate of convergence for vanishing viscosity of the viscous Hamilton--Jacobi equation with the Dirichlet boundary condition (see \cite{crandall1984,evans_adjoint_2010,fleming_convergence_1961,Tran2011}). Let $g(x) = u(x)$ for $x\in \partial\Omega$. Let $v^\varepsilon\in \mathrm{C}^2(\Omega)\cap \mathrm{C}(\overline{\Omega})$ be the unique viscosity solution to
\begin{equation*}
\left\{
    \begin{aligned}
          v^\varepsilon(x) + |Dv^\varepsilon(x)|^p - f(x) - \varepsilon \Delta v^\varepsilon(x) &= 0 \,\qquad\text{in}\;\Omega,\\
        v^\varepsilon(x) &= g(x) \ \ \  \text{on}\;\partial\Omega.
    \end{aligned}\right.
\end{equation*}
It is well known that $v^\varepsilon \to u$. Furthermore, there exists a positive constant $C$ independent of $\varepsilon\in (0,1)$ such that
\begin{equation}\label{e:cp1}
     |v^\varepsilon(x)  - u(x)| \leq C\sqrt{\varepsilon} \qquad\text{for}\;x\in \overline{\Omega}.
\end{equation}
By the comparison principle for \eqref{eq:PDEeps}, we have
\begin{equation}\label{e:cp2}
    v^\varepsilon(x)\leq u^\varepsilon(x) \qquad\text{for}\;x\in \Omega.
\end{equation}
From \eqref{e:cp1} and \eqref{e:cp2}, we obtain the lower bound \eqref{e:lower1}.
\end{proof}

\begin{rem} If $f\equiv C$ in $\Omega$ for some constant $C$, then \eqref{e:lower1} can be improved to $0\leq u^\varepsilon - u$, since in this case $u \equiv C$.
\end{rem}
%We state some observations regarding the property of the solution $u$ to \eqref{eq:PDE0}.

\begin{lem}\label{lem:f=0} Assume $f\geq 0$ in $\Omega$. Then $u(x) = 0$ if and only if $f(x) = 0$. In particular, $f \equiv 0$ implies $u \equiv 0$.
\end{lem}
\begin{proof} From Lemma \ref{lem:lower-bound}, we know $u\geq 0$ in $\overline{\Omega}$. The optimal control formula for the solution $u$ to \eqref{eq:PDE0} is (see \cite{Bardi1997,tran_hamilton-jacobi_2021})
\begin{equation}\label{e:foru}
    u(x) = \inf \left\lbrace \int_0^\infty e^{-  s}\Big(C_p|\dot{\eta}(s)|^{q} +f(\eta(s))\Big)ds: \eta\in \mathrm{AC}([0,\infty);\overline{\Omega}), \eta(0) = x\right\rbrace,
\end{equation}
where
\begin{equation*}
    C_p = \left(qp^\frac{1}{p-1}\right)^{-1} \qquad\text{and}\qquad \frac{1}{p} + \frac{1}{q} = 1.
\end{equation*}
%\textcolor{red}{(Here, the constant $c=\frac{1}{qp^\frac{1}{p-1}}$? )}

Let $x\in \overline{\Omega}$ such that $f(x) = 0$. We can choose $\eta(s) = x$ for all $s\in [0,\infty)$ as an admissible path in \eqref{e:foru} to obtain that
\begin{equation*}
    u(x) \leq  \int_0^\infty e^{-  s}\Big(C_p|\dot{\eta}(s)|^{q} +f(\eta(s))\Big)ds = f(x) =  0.
\end{equation*}
As a consequence, $f\equiv 0$ implies $u\equiv 0$.

It is not hard to prove the converse by contradiction. Suppose $u(x_0) =0$ and $f(x_0)>0$. Then there exists $\varepsilon$, $\delta >0$ such that $f(x) > \varepsilon$ for all $x \in B_{\delta} (x_0)$. Let $\eta \in AC([0,\infty);\overline{\Omega})$ such that $\eta(0)=x_0$ and t be the time that $\eta$ first hits $\partial B_{\delta}(x_0)$. Note that $t$ could be $+\infty$. Then
\begin{equation*}
    \begin{aligned}
        \int_0^\infty e^{-s} \left(|\dot{\eta}(s)|^q +f(s )\right) ds &\geq \int_0^t e^{-s} \left(|\dot{\eta}(s)|^q +f(s )\right) ds\\
        &\geq  \frac{1}{e^t t^{q-1}}\left| \int_0^t \dot{\eta} (s) ds\right|^q+ \varepsilon \left(1-e^{-t} \right)\\
        &\geq \frac{\delta^q}{e^t t^{q-1}}+\varepsilon \left(1-e^{-t} \right),
    \end{aligned}
\end{equation*}
where we used Jensen's inequality in the second line.
This implies $u(x_0)>0$ since $q \geq 2$, which is a contradiction.
\end{proof}

\noindent The following lemma is about a crucial estimate that will be used. It is a refined construction of a supersolution for \eqref{eq:PDEeps}.
\begin{lem}\label{lem:super_refined} Let $\delta_0$ be defined as in \eqref{def:delta_0}. There exist positive constants $\nu = \nu(\delta_0)> 1$ and $C_\nu =\mathcal{O}\left(\delta_0^{-(\alpha+2)}\right)$ such that
\begin{equation}\label{e:superwa}
w(x) = \begin{cases}
    \displaystyle\frac{\nu C_\alpha \varepsilon^{\alpha+1}}{d(x)^\alpha} + \max f + C_\nu \varepsilon^{\alpha+2}, \qquad\;\;\,  p<2,\vspace{0.2cm}\\
    \displaystyle\nu \varepsilon \log\left(\frac{1}{d(x)}\right) + \max f+ C_\nu \varepsilon^2, \qquad  p=2,
\end{cases}
\end{equation}
is a supersolution of \eqref{eq:PDEeps} in $\Omega$.
%\item[(ii)] If $x\mapsto \mathrm{dist}(x,\partial\Omega)$ is continuously differentiable with $|D \mathrm{dist}(x,\partial\Omega)| = 1$ everywhere in $\Omega$ then we can take $C_\nu = 0$ in \eqref{e:superwa}.
\end{lem}
%\textcolor{orange}{(Why is $\nu \in (1,2)$? I think when $\alpha$ is large, it can be larger than 2. But $\nu$ is certainly bounded.---Oh, you can choose $\delta_0$ so small that $\nu \in (1,2)$. But then $\delta_0$ depends on $\alpha$.)} -  \textcolor{red}{I think dependence on $\alpha$ is fine, as we fix $p$ to start with.} \textcolor{orange}{- Do we include that in the definition of $\delta_0$?} --  \textcolor{red}{I will just write $\nu$ for now, not requiring $\nu \leq 2$ anymore.}

\begin{proof} Let us first consider $1<p<2$. Recall from Theorem \ref{thm:wellposed1<p<2} that $C_\alpha^p \alpha^p = C_\alpha \alpha (\alpha+1)$ and $p(\alpha+1) = \alpha+2$. Compute
\begin{equation*}
    |D w(x)|^p = \nu^p\frac{(C_\alpha\alpha)^p\varepsilon^{p(\alpha+1)}}{d(x)^{p(\alpha+1)}} \left| Dd(x) \right|^p= \nu^p\frac{C_\alpha \alpha(\alpha+1)\varepsilon^{\alpha+2}}{d(x)^{\alpha+2}}  \left| Dd(x) \right|^p
\end{equation*}
and
\begin{equation*}
    \varepsilon\Delta w(x) = \nu\frac{C_\alpha\alpha(\alpha+1)\varepsilon^{\alpha+2}}{d(x)^{\alpha+2}} \left| Dd(x)\right|^2- \nu\frac{C_\alpha\alpha\varepsilon^{\alpha+2}\Delta d(x)}{d(x)^{\alpha+1}}.
\end{equation*}
We have
\begin{equation*}
    \begin{aligned}
        \mathcal{L}^\varepsilon\left[ w \right] =  &\frac{\nu C_\alpha\varepsilon^{\alpha+1}}{d(x)^\alpha} + \max f - f(x) + C_\nu \varepsilon^{\alpha+2} \\&+ \frac{C_\alpha \alpha (\alpha+1)\varepsilon^{\alpha+2}}{d(x)^{\alpha+2}}\left[\nu^p\left| Dd(x) \right|^p-\nu \left| Dd(x)\right|^2+\nu\frac{d(x)\Delta d(x)}{\alpha+1}\right] .
    \end{aligned}
\end{equation*}
\paragraph{\textbf{Case 1.}} If $0< d(x)\leq \delta_0$, we have $|D d(x)| = 1$. Recall that $K_2= \Vert \Delta d\Vert_{L^\infty}$ and observe
\begin{equation*}
    \left|\frac{d(x)\Delta d(x)}{\alpha+1}\right| \leq \frac{\delta_0\Vert \Delta d\Vert_{L^\infty}}{\alpha+1} \leq \frac{K_2\delta_0}{\alpha+1}\leq K_2\delta_0.
\end{equation*}
Therefore,
\begin{equation}\label{e:choose_nu}
    \begin{split}
        \nu^p-\nu +\nu\frac{d(x) \Delta d(x)}{(\alpha+1)} \geq \nu^p - \nu - \nu K_2\delta_0 = \nu\Big(\nu^{p-1} - (1+K_2\delta_0)\Big).
    \end{split}
\end{equation}
We will choose $\nu$ as follows. For $\gamma>1$, we have the inequality
\begin{equation}\label{e:ineq}
    \Big||x+y|^\gamma - |x|^\gamma\Big|\leq \gamma\Big(|x|+|y|\Big)^{\gamma-1}|y|
\end{equation}
for $x,y\in \mathbb{R}$, which implies that
\begin{equation*}
     0 \leq (1+K_2\delta_0)^{\alpha+1} - 1 \leq\underbrace{(\alpha+1)\left(1+K_2\delta_0\right)^\alpha K_2}_{C_2}\delta_0.
\end{equation*}
Hence, $(1+K_2\delta_0)^{\alpha+1} \leq 1 + C_2\delta_0$. Since $\alpha+1 = \frac{1}{p-1}$,
\begin{equation}\label{e:cru1}
    (1+K_2\delta_0) \leq (1+C_2\delta_0)^{\frac{1}{\alpha+1}} = (1+C_2\delta_0)^{p-1}.
\end{equation}
Choose $\nu = 1+C_2\delta_0$ in \eqref{e:choose_nu} and we obtain $\mathcal{L}[w]\geq 0$ in $\{x\in \Omega_\delta: \delta <d(x)\leq \delta_0\}$. \\

\paragraph{\textbf{Case 2.}} If $d(x)\geq \delta_0$, recall that $K_0 = \Vert d\Vert_{L^\infty}$ and $K_1 = \Vert D d\Vert_{L^\infty}$. And we have
\begin{align*}
    \mathcal{L}[w] = & \frac{\nu C_\alpha \varepsilon^{\alpha+1}}{d(x)^\alpha} + \max_{\Omega} f - f(x)\\
    &+  \nu^p\frac{C_\alpha\alpha(\alpha+1)\varepsilon^{\alpha+2}}{d(x)^{\alpha+2}}|D d(x)|^p - \nu \frac{C_\alpha \alpha(\alpha+1)\varepsilon^{\alpha+2}}{d(x)^{\alpha+2}}|D d(x)|^2
    + \nu \frac{C_\alpha \alpha \varepsilon^{\alpha+2}\Delta d(x)}{d(x)^{\alpha+1}} + C_\nu \varepsilon^{\alpha+2}\\
    \geq &\frac{C_\alpha \alpha(\alpha+1)\varepsilon^{\alpha+2}}{d(x)^{\alpha+2}}\left(\nu^p|D d(x)|^p - \nu |D d(x)|^2 + \nu \frac{d(x) \Delta d(x)}{\alpha+1}\right) + C_\nu \varepsilon^{\alpha+2}\\
    \geq & \left[C_\nu - C_3\left(\frac{1}{\delta_0}\right)^{\alpha+2}\right]\varepsilon^{\alpha+2},
\end{align*}
where
\begin{equation*}
    C_3 = C_\alpha\alpha(\alpha+1) \left(\nu^pK_1^p + \nu K_1^2 + \nu \frac{K_0K_2}{\alpha+1}\right).
\end{equation*}
We can choose $C_\nu = C_3\delta_0^{-(\alpha+2)}$ to obtain $\mathcal{L}[w]\geq 0$ in $\{x\in \Omega_\delta:d(x)\geq \delta_0\}$.
\smallskip

\noindent
If $p=2$, then $\alpha = 0$. We can easily see that the similar calculation holds true with $\nu := 1+K_2\delta_0$ and $C_\nu := \delta_0^{-2}\nu (\nu K_1^2 + K_1^2+ K_0K_2)$.
\end{proof}

Now we begin to present the rate of convergence for the special case where $f = C_f$ in $\Omega$ for some constant $C_f$.

\begin{thm}[Constant data]\label{thm:rate_doubling0} Assume $f\equiv C_f$ in $\Omega$. Let $u^\varepsilon$ be the unique solution to \eqref{eq:PDEeps} and $u \equiv C_f$ be the unique solution to \eqref{eq:PDE0}. Then there exists a constant $C$ independent of $\varepsilon\in(0,1)$ such that
    \begin{equation*}
    \begin{split}
    &0\leq u^\varepsilon(x) - u(x)\leq C \left(\frac{ \varepsilon^{\alpha+1}}{d(x)^\alpha} + \frac{\varepsilon^{\alpha+2}}{\delta_{0,\Omega}^{\alpha+2}}\right),  \qquad\qquad \;\;\; \text{if}\; 1<p<2,\\
    &0\leq u^\varepsilon(x) - u(x)\leq C \left(\varepsilon \mathrm{log}\left(\frac{1}{d(x)}\right) + \frac{\varepsilon^{2}}{\delta_{0,\Omega}^{2}}\right),  \qquad \text{if}\; p=2,
    \end{split}
\end{equation*}
for $x\in \Omega$, where $\delta_{0,\Omega}$ is defined as in \eqref{def:delta_0}. In particular,
\begin{itemize}
    \item[(i)] if $1<p<2$, we have $C_f\leq u^\varepsilon(x)\leq C_f + C\varepsilon$ for $x\in \Omega_\varepsilon = \{x\in \Omega: \mathrm{dist}(x,\partial\Omega) \geq \varepsilon\}$, and
    \item[(ii)] for any $K\subset\subset \Omega$, there holds $\Vert u^\varepsilon - u\Vert_{L^\infty(K)} \leq C\varepsilon^{\alpha+1}$ .
\end{itemize}
\end{thm}

\begin{proof} Lemma \ref{lem:f=0} implies $u \equiv C_f$ in $\Omega$. And Lemma \ref{lem:lower-bound} tells us $u^\varepsilon-u=u^\varepsilon-C_f \geq 0$.By the comparison principle of \eqref{eq:PDEeps} and Lemma \ref{lem:super_refined}, the conclusion follows.
\end{proof}

\begin{rem} The conclusion of Theorem \ref{thm:rate_doubling0} also holds if $f = C_f+  \mathcal{O}(\varepsilon^{\beta})$ for $\beta \geq \alpha+1$.
\end{rem}

Even this special case (Theorem \ref{thm:rate_doubling0}) is new in the literature. As an immediate consequence, we obtain the rate of convergence on any compact subset that is disjoint from the support of $f$.

\begin{cor} Assume $f$ is Lipschitz with compact support and $K$ is a connected compact subset of $\Omega$ that is disjoint from $\mathrm{supp}(f)$. Then there exists a constant $C = C(K)$ independent of $\varepsilon\in (0,1)$ such that
\begin{equation*}
     \Vert u^\varepsilon - u\Vert_{L^\infty(K)} \leq C\varepsilon^{\alpha+1}.
\end{equation*}
\end{cor}
\begin{proof} We choose an open, bounded and connected set $U$ such that $\partial U$ is $\rmC^2$ and $K\subset\subset U \subset\subset \Omega$. Let $w^\varepsilon$ be the solution to \eqref{eq:PDEeps} with $\Omega$ replaced by $U$. Then by Theorem \ref{thm:rate_doubling0}, we have
\begin{equation*}
    0\leq w^\varepsilon(x)\leq C\left(\varepsilon^{\alpha+1} + \varepsilon^{\alpha+2}\right), \qquad x\in K,
\end{equation*}
where $C$ depends on $\mathrm{dist}(K, \partial U)$ and $U$. Recall that $u = 0$ outside the support of $f$. By the comparison principle in $U$, we see that $u^\varepsilon \leq w^\varepsilon$ and thus the conclusion follows.
\end{proof}

For the general result of nonnegative compactly supported data, we have the following theorem.

\begin{thm}[Nonnegative compactly supported data]\label{thm:rate_doubling1} Assume that $f$ is nonnegative and Lipschitz with compact support in $\Omega_{\kappa}$. Let $u^\varepsilon$ be the unique solution to \eqref{eq:PDEeps} and $u$ be the unique solution to \eqref{eq:PDE0}. Then there exists a constant $C$ independent of $\varepsilon\in(0,1)$ and $\kappa$ such that
\begin{align}
   &-C \sqrt{\varepsilon}\leq  u^\varepsilon(x) - u(x) \leq C\left(\sqrt{\varepsilon}+\left(\frac{\varepsilon}{\kappa}\right)^{\alpha+2}\right) + \frac{\nu C_\alpha  \varepsilon^{\alpha+1}}{d(x)^\alpha}, \quad\; \qquad p<2,\label{eq:cp1}\\
   &-C\sqrt{\varepsilon}\leq u^\varepsilon(x) - u(x) \leq C \left(\sqrt{\varepsilon}+ \left(\frac{\varepsilon}{\kappa}\right)^2\right) +  \nu \varepsilon \log\left(\frac{1}{d(x)}\right), \qquad p=2, \label{eq:cp2}
\end{align}
for any $x\in \Omega$. As a consequence, $|u^\varepsilon(x)-u(x)|\leq C\sqrt{\varepsilon}$ for all $x\in \Omega$ with $d(x)\geq \varepsilon$.
%Furthermore, on any compact subset $K\subset\Omega\backslash \mathrm{supp}(f)$, $|u^\varepsilon - u| = \mathcal{O}(\varepsilon)$.
\end{thm}

We state the following lemma as a preparation.

\begin{lem} Let $0<\kappa < \delta_0$ and $U_\kappa = \big\{x\in \Omega: 0<\mathrm{dist}(x,\partial\Omega) < \kappa\big\} = \Omega\backslash \overline{\Omega}_\kappa$. There holds
\begin{equation*}
    \mathrm{dist}(x,\partial\Omega_\kappa) = \kappa - \mathrm{dist}(x,\partial\Omega) \qquad\text{for all}\;x\in U_\kappa.
\end{equation*}
As a consequence, $x\mapsto \mathrm{dist}(x,\partial U_\kappa) = \min\big\{\mathrm{dist}(x,\partial \Omega_k),\mathrm{dist}(x,\partial \Omega)\big\}$ is twice continuously differentiable for $x\in \Omega\backslash \overline{\Omega}_{\kappa/2}$. Hence, we can choose
\begin{equation}\label{e:delta_kappa}
    \delta_{0,U_\kappa} \geq \frac{\kappa}{4}
\end{equation}
where $\delta_{0,\Omega}$ is defined as in \eqref{def:delta_0}.
\end{lem}

\begin{proof} By the definition of $\delta_0 = \delta_{0,\Omega}$, we have $d(x) = \mathrm{dist}(x,\partial\Omega)$ is twice continuously differentiable in the region $U_{\delta_0} = \Omega\backslash \overline{\Omega}_{\delta_0}$. The proof follows from \cite[p. 355]{gilbarg_elliptic_2001}.
\end{proof}

%\textcolor{red}{"Furthermore, on any compact subset $K\subset\Omega\backslash \mathrm{supp}(f)$, $|u^\varepsilon - u| = \mathcal{O}(\varepsilon)$."--It sounds like the constant inside $O(\varepsilon)$ is uniform in $K$, and do we require anything about the support of $f$?...(I will send a picture.)}

\begin{proof}[Proof of Theorem \ref{thm:rate_doubling1}] Without loss of generality, assume that $f$ is supported in $\Omega_\kappa$ where $0<\kappa < \delta_{0}$.
Let $g_\kappa = u^\varepsilon$ on $\partial\Omega_{\kappa}$. Then the solution $u^\varepsilon$ of \eqref{eq:PDEeps} also solves
\begin{equation*}
    \left\{
  \begin{aligned}
  u^\varepsilon(x) + |Du^\varepsilon(x)|^p-\varepsilon \Delta u^\varepsilon(x) &=0 \;\qquad \text{in } U_\kappa ,\\
  u^\varepsilon(x) &= +\infty \quad \text{on } \partial \Omega,\\
  u^\varepsilon(x) &= g_\kappa \;\;\quad \text{on } \partial \Omega_{\kappa},
    \end{aligned}
\right.
\end{equation*}
in the annulus $U_\kappa= \Omega \setminus \overline{\Omega}_{\kappa} = \{x\in \Omega: 0< d(x) < \kappa\}$. Let $\tilde{u}^\varepsilon\in \mathrm{C}^2(U_\kappa)$ be the solution to the following problem
\begin{equation*}
    \left\{
        \begin{aligned}
            \tilde{u}^\varepsilon(x) + |D\tilde{u}^\varepsilon(x)|^p-\varepsilon \Delta \tilde{u}^\varepsilon(x) &=0 \;\qquad \text{in } U_\kappa ,\\
            \tilde{u}^\varepsilon(x) &= +\infty \quad \text{on } \partial U_\kappa = \partial \Omega\cup \partial \Omega_{\kappa},
        \end{aligned}
    \right.
\end{equation*}
whose existence is guaranteed by Theorem \ref{thm:wellposed1<p<2}.
Here the boundary condition is understood in the sense that $\tilde{u}^\varepsilon(x)\to \infty$ as $d_\kappa(x)\to 0$, where $d_\kappa(\cdot)$ is the distance function from the boundary of the annulus $U_\kappa$, i.e.,
\begin{equation*}
    d_\kappa(x) = \min \big\lbrace \mathrm{dist}(x,\partial \Omega_\kappa),\mathrm{dist}(x,\partial\Omega)  \big\rbrace \leq d(x) \qquad\text{for}\;x\in U_\kappa.
\end{equation*}
Since $f = 0$ in $\overline{U}_\kappa$, by Lemma \ref{lem:f=0}, $u=0$ in $\overline{U}_\kappa$. Hence, $u$ is also the unique state-constraint solution to
\begin{equation*}
    \left\{
        \begin{aligned}
            u(x)+ |Du(x)|^p &=0 \quad \text{in } U_\kappa ,\\
            u(x)+ |Du(x)|^p &\geq 0 \quad \text{on } \partial U_\kappa = \partial \Omega \cup\partial \Omega_{\kappa}.
        \end{aligned}
    \right.
\end{equation*}
The vanishing viscosity of $\tilde{u}^\varepsilon \to 0$ in $U_\kappa$ can be quantified by Theorem \ref{thm:rate_doubling0}, which gives us
\begin{equation*}
\begin{split}
    &0\leq \tilde{u}^\varepsilon(x) \leq \frac{\nu C_\alpha \varepsilon^{\alpha+1}}{d_\kappa(x)^\alpha}+C_3\left(\frac{\varepsilon}{\delta_{0,U_\kappa}}\right)^{\alpha+2}\qquad\quad\;\text{for}\;p<2,\\
    &0\leq \tilde{u}^\varepsilon(x) \leq \nu \varepsilon \log\left(\frac{1}{d_\kappa(x)}\right)+C\left(\frac{\varepsilon}{\delta_{0,U_\kappa}}\right)^{2}\qquad\text{for}\;p=2,
\end{split}
\end{equation*}
for $x\in U_\kappa$. From \eqref{e:delta_kappa} and the comparison principle in $U_\kappa$, we have
\begin{align}
    &0\leq u^\varepsilon(x) \leq \tilde{u}^\varepsilon(x)  \leq \frac{\nu C_\alpha\varepsilon^{\alpha+1}}{d_\kappa(x)^{\alpha}} + C_3\left(\frac{4\varepsilon}{\kappa}\right)^{\alpha+2} \qquad\quad\;\text{for}\;p<2, \qquad \label{annulus2}\\
    &0\leq u^\varepsilon(x)\leq \tilde{u}^\varepsilon(x) \leq \nu \varepsilon \log\left(\frac{1}{d_\kappa(x)}\right)+C\left(\frac{4\varepsilon}{\kappa}\right)^{2}\qquad\text{for}\;p=2,\label{annulus2p=2}
\end{align}
for $x\in U_\kappa$. We proceed with the doubling variable method. For $p<2$, consider the auxiliary functional
\begin{equation*}
    \Phi(x,y)= u^\varepsilon(x) - u(y) -\frac{C_0|x-y|^2}{\sigma} - \frac{\nu C_\alpha \varepsilon^{\alpha +1}}{d(x)^\alpha}, \qquad (x,y)\in \overline{\Omega}\times \overline{\Omega},
\end{equation*}
where $C_0$ is the Lipschitz constant of $u$ from \eqref{e:C0}, $\sigma\in (0,1)$. The fact that $\displaystyle d(x)^\alpha u^\varepsilon(x) \to C_\alpha \varepsilon^{\alpha+1}$ as $d(x) \to 0^+$ implies
\begin{equation*}
    \max_{(x,y) \in \overline{\Omega} \times \overline{\Omega}} \Phi(x,y) = \Phi(x_\sigma, y_\sigma) \qquad\text{for some}\;(x_\sigma,y_\sigma) \in \Omega \times \overline{\Omega}.
\end{equation*}
From $\Phi(x_\sigma, y_\sigma) \geq \Phi(x_\sigma, x_\sigma)$, we can deduce that
\begin{equation}\label{e:sigma}
    \left| x_\sigma - y_\sigma \right| \leq \sigma.
\end{equation}
If $ d(x_\sigma) \geq \frac{1}{2}\kappa$, since $x\mapsto \Phi(x,y_\sigma)$ has a maximum over $\Omega$ at $x=x_\sigma$, the subsolution test for $u^\varepsilon(x)$ gives us
\begin{align}\label{e:subsln}
    &u^\varepsilon(x_\sigma) + \left|\frac{2C_0(x_\sigma - y_\sigma)}{\sigma} -  \frac{\nu C_\alpha\alpha \varepsilon^{\alpha+1} D d(x_\sigma)}{d(x_\sigma)^{\alpha+1}}\right|^p - f(x_\sigma)\nonumber\\
    &\qquad -\varepsilon\left(\frac{2nC_0}{\sigma}+ \frac{\nu C_\alpha\alpha(\alpha+1) \varepsilon^{\alpha+1}|D d(x_\sigma)|^2}{d(x_\sigma)^{\alpha+2}} - \frac{\nu C_\alpha\alpha \varepsilon^{\alpha+1}\Delta d(x_\sigma)}{d(x_\sigma)^{\alpha+1}}\right) \leq 0.
\end{align}
Since $y\mapsto \Phi(x_\sigma,y)$ has a maximum over $\overline{\Omega}$ at $y = y_\sigma$, the supersolution test for $u(y)$ gives us
\begin{align}\label{e:supersln}
    u(y_\sigma) + \left|\frac{2C_0(x_\sigma - y_\sigma)}{\sigma}\right|^p - f(y_\sigma) \geq 0.
\end{align}
For simplicity, define
\begin{equation*}
    \xi_\sigma := \frac{2C_0(x_\sigma - y_\sigma)}{\sigma} \qquad\text{and}\qquad \zeta_\sigma :=- \frac{\nu C_\alpha\alpha \varepsilon^{\alpha+1} D d(x_\sigma)}{d(x_\sigma)^{\alpha+1}}.
\end{equation*}
From \eqref{e:sigma} and $d(x_\sigma) \geq \frac{1}{2}\kappa$,
\begin{equation*}
    |\xi_\sigma|\leq 2C_0, \qquad\text{and}\qquad |\zeta_\sigma| \leq \nu K_1 C_\alpha\alpha \left(\frac{\varepsilon}{d(x_\sigma)}\right)^{\alpha+1} \leq \nu K_1 C_\alpha \alpha  \left(\frac{2\varepsilon}{\kappa}\right)^{\alpha+1}.
\end{equation*}
Using the inequality \eqref{e:ineq} with $\gamma = p > 1$, we deduce that
\begin{align}\label{e:estia}
    \Big||\xi_\sigma +\zeta_\sigma|^p - |\xi_\sigma|^p \Big| &\leq p\Big(|\xi_\sigma|+|\zeta_\sigma|\Big)^{p-1}|\zeta_\sigma|\nonumber\\
    &\leq p\left[2C_0+\nu K_1 C_\alpha\alpha \left(\frac{2\varepsilon}{\kappa}\right)^{\alpha+1}\right]^{p-1}\nu K_1 C_\alpha\alpha \left(\frac{2\varepsilon}{\kappa}\right)^{\alpha+1}.
\end{align}
Combine \eqref{e:estia} together with \eqref{e:subsln}, \eqref{e:supersln} and $|f(x_\sigma) - f(y_\sigma)|\leq C|x_\sigma - y_\sigma|\leq C\sigma$ to obtain
\begin{align*}
     u^\varepsilon(x_\sigma) - u(y_\sigma) \leq & p\left(2C_0+\nu K_1 C_\alpha \alpha\left( \frac{2\varepsilon}{\kappa}\right)^{\alpha+1}\right)^{p-1}\nu K_1 C_\alpha \alpha \left(\frac{2\varepsilon}{\kappa}\right)^{\alpha+1} + C\sigma\nonumber\\
    &+2nC_0\left(\frac{\varepsilon}{\sigma}\right) + \nu K_1^2C_\alpha \alpha(\alpha+1)\left(\frac{2\varepsilon}{\kappa}\right)^{\alpha+2} + \nu K_2 C_\alpha \alpha  \left(\frac{2\varepsilon}{\kappa}\right)^{\alpha+1}\varepsilon\nonumber\\
    \leq & C\left[\sigma + \frac{\varepsilon}{\sigma} + \left(1+\left(\frac{\varepsilon}{\kappa}\right)^{\alpha+1}\right)^{p-1}\left(\frac{\varepsilon}{\kappa}\right)^{\alpha+1} + \left(\frac{\varepsilon}{\kappa}\right)^{\alpha+2} \right].
    %&\leq C\sqrt{\varepsilon}
\end{align*}
By the fact that $(1+x)^\gamma \leq 1+x^\gamma$ for $x\in [0,1]$ and $\gamma\in [0,1]$, we know
\begin{equation*}
    \left(1+\left(\frac{\varepsilon}{\kappa}\right)^{\alpha+1}\right)^{p-1} \leq 1+ \left(\frac{\varepsilon}{\kappa}\right),
\end{equation*}
as $0<p-1\leq 1$.
Therefore,
\begin{equation*}
     u^\varepsilon(x_\sigma) - u(y_\sigma) \leq C\left[\sigma + \frac{\varepsilon}{\sigma} + \left(\frac{\varepsilon}{\kappa}\right)^{\alpha+1} + \left(\frac{\varepsilon}{\kappa}\right)^{\alpha+2}\right],
\end{equation*}
where $C$ is independent of $\kappa$ and $\varepsilon$. Now choose $\sigma = \sqrt{\varepsilon}$ to get (with $\kappa$ fixed)
\begin{equation}\label{e:final1}
    \Phi(x_\sigma,y_\sigma) \leq u^\varepsilon(x_\sigma) - u(y_\sigma) \leq C\sqrt{\varepsilon}.
\end{equation}
If $d(x_\sigma) < \frac{1}{2}\kappa$, then $x_\sigma \in U_\kappa$ and furthermore $\mathrm{dist}(x_\sigma,\partial \Omega_\kappa) > \frac{1}{2}\kappa$. Indeed, for any $y\in\partial\Omega$ and $z\in \partial\Omega_k$, we have $|x_\sigma - z| + |x_\sigma - y| \geq |y-z|$. Taking the infimum over all $y\in \partial\Omega$, we deduce that
\begin{equation*}
    |x_\sigma - z| + d(x_\sigma) \geq \inf_{y\in \partial \Omega}|y-z| = d(z) = \kappa
\end{equation*}
since $z\in \partial\Omega_k = \{x \in \Omega: d(x) = \kappa\}$. Thus, $|x_\sigma - z| \geq \kappa - d(x_\sigma) > \frac{1}{2}\kappa$ for all $z\in \partial\Omega_k$, which implies that $\mathrm{dist}(x_\sigma,\partial\Omega_k)>\frac{1}{2}\kappa$ and hence $d_\kappa(x_\sigma) = d(x_\sigma)$. By \eqref{annulus2} and the fact that $u \geq 0$, we have
\begin{align}
    \Phi(x_\sigma, y_\sigma)\leq u^\varepsilon(x_\sigma) - \frac{\nu C_\alpha \varepsilon^{\alpha+1}}{d(x_\sigma)^\alpha} \leq C_3\left(\frac{4\varepsilon}{\kappa}\right)^{\alpha+2} \label{e:final2}.
\end{align}
%\textcolor{red}{(I just read this one more time and feel confused here as the RHS should be $\frac{\nu C_\alpha \epsilon^{\alpha + 1}}{d(x_\sigma)}+C_3(\frac{4\epsilon}{\kappa})^{\alpha+2}$. But then the following argument doesn't make sense. I think the inequality should be something like
%\begin{equation*}
%\begin{aligned}
%    \Phi(x_\sigma, y_\sigma) &\leq u^\varepsilon (x_\sigma)- \frac{\theta  C_\alpha  \varepsilon^{\alpha+1}}{d(x_\sigma)^\alpha}\nonumber\\
 %   &= u^\varepsilon (x_\sigma)-\frac{\nu C_\alpha  \varepsilon^{\alpha+1}}{d_\kappa(x_\sigma)^\alpha}
 %   \leq C_3\left(\frac{4\varepsilon}{\kappa}\right)^{\alpha+2}
%\end{aligned}
%\end{equation*}
%where we choose $\theta = \nu$, similar to the proof in draft01. The conclusion isn't much more different.)}
\noindent
Since $\Phi(x,x) \leq \Phi(x_\sigma,y_\sigma)$ for all $x\in \Omega$, we obtain from \eqref{e:final1} and \eqref{e:final2} that
\begin{equation*}
    u^\varepsilon(x)-u(x)-\frac{\nu C_\alpha  \varepsilon^{\alpha+1}}{d(x)^\alpha} \leq C\sqrt{\varepsilon} +C_3\left(\frac{4\varepsilon}{\kappa}\right)^{\alpha+2}
\end{equation*}
and thus \eqref{eq:cp1} follows.

\noindent For $p=2$, we consider instead the functional
\begin{equation*}
    \Phi(x,y) = u^\varepsilon(x) - u(y) - \frac{C_0|x-y|^2}{\sigma} - \nu \varepsilon \mathrm{log}\left(\frac{1}{d(x)}\right), \qquad (x,y)\in \overline{\Omega}\times \overline{\Omega}.
\end{equation*}
Similar to the previous case where $1<p<2$, the maximum of $\Phi$ occurs at some point $(x_\sigma,y_\sigma)\in \Omega\times\overline{\Omega}$ and $|x_\sigma - y_\sigma|\leq \sigma$. If $d(x_\sigma)\geq \frac{1}{2}\kappa$, by the subsolution test for $u^\varepsilon(x)$, we have
\begin{align}\label{e:subslnp=2}
    u^\varepsilon(x_\sigma)&+ \left|\frac{2C_0(x_\sigma - y_\sigma)}{\sigma} - \nu \varepsilon \frac{D d(x_\sigma)}{d(x_\sigma)}\right|^2  -f(x_\sigma) \nonumber\\
    &- 2nC_0\left(\frac{\varepsilon}{\sigma}\right) - \nu |D d(x_\sigma)|^2 \left(\frac{\varepsilon}{d(x_\sigma)}\right)^2  + \nu \Delta d(x_\sigma)\left(\frac{\varepsilon^2}{d(x_\sigma)}\right) \leq 0.
\end{align}
By the supersolution test for $u(y)$, we have
\begin{equation}\label{e:superslnp=2}
    u(y_\sigma) + \left|\frac{2C_0(x_\sigma - y_\sigma)}{\sigma}\right|^2 - f(y_\sigma) \geq 0.
\end{equation}
Subtract \eqref{e:superslnp=2} from \eqref{e:subslnp=2} to get
\begin{align*}
    u^\varepsilon(x_\sigma) - u(y_\sigma) \leq & \left(4C_0+ \nu\varepsilon\frac{D d(x_\sigma)}{d(x_\sigma)}\right)\left(\nu \varepsilon \frac{ D d(x_\sigma)}{d(x_\sigma)}\right) \\
    &+ C\sigma + 2nC_0 \left(\frac{\varepsilon}{\sigma}\right)+ \nu |D d(x_\sigma)|^2 \left(\frac{\varepsilon}{d(x_\sigma)}\right)^2 + \nu|\Delta d(x_\sigma)| \frac{\varepsilon^2}{d(x_\sigma)}.
\end{align*}
Using $d(x_\sigma)\geq \frac{1}{2}\kappa$ and bounds on $d(x)$ in \eqref{boundond}, we see that
\begin{align}\label{p=2a}
    \Phi(x_\sigma,y_\sigma)&\leq u^\varepsilon(x_\sigma) - u(y_\sigma)\nonumber \\
    &\leq 4K_1^2\nu(1+\nu)\left(\frac{\varepsilon}{\kappa}\right)^2 + C\sigma + 2nC_0\left( \frac{\varepsilon}{\sigma}\right) + 2\nu(K_2\varepsilon+4C_0K_1)\left(\frac{\varepsilon}{\kappa}\right)\nonumber\\
    &\leq C\left(\sigma+\frac{\varepsilon}{\sigma} + \frac{\varepsilon}{\kappa} + \left(\frac{\varepsilon}{\kappa}\right)^2\right) \leq C\sqrt{\varepsilon}
\end{align}
if we choose $\sigma = \sqrt{\varepsilon}$.

\noindent
If $d(x_\sigma)<\frac{1}{2}\kappa$, then $x_\sigma\in U_\kappa$. Again, we have $d_\kappa(x_\sigma) = d(x_\sigma)$ and from \eqref{annulus2p=2}
\begin{equation}\label{p=2b}
    \Phi(x_\sigma,y_\sigma) \leq u^\varepsilon(x_\sigma) - \nu \varepsilon \log\left(\frac{1}{d(x_\sigma)}\right) \leq  C\left(\frac{4\varepsilon}{\kappa}\right)^2. %\leq C\left(\varepsilon+\varepsilon\,|\log(\kappa)| + \left(\frac{\varepsilon}{\kappa}\right)^2\right).
\end{equation}
Since $\Phi(x,x)\leq \Phi(x_\sigma,y_\sigma)$ for $x\in \Omega$, we obtain from \eqref{p=2a} and \eqref{p=2b} that
\begin{equation*}
    u^\varepsilon(x) - u(x) - \nu\varepsilon\log\left(\frac{1}{d(x)}\right) \leq C\sqrt{\varepsilon} +C\left(\frac{4\varepsilon}{\kappa}\right)^2
\end{equation*}
and thus \eqref{eq:cp2} follows.
%Similar to the case where $1<p<2$, \eqref{annulus2p=2} gives us a better rate $\mathcal{O}(\varepsilon)$ in any compact set $K$ that is disjoint from $\mathrm{supp}(f)$. Thus, the conclusion follows.
\end{proof}

%\begin{rem} From \cite{Lasry1989} we know that, if $p\in \left(\frac{3}{2},2\right]$ then $u^\varepsilon(x) - \frac{C_\alpha\varepsilon^{\alpha+1}}{d(x)^\alpha}$ is uniformly bounded in $\Omega$. Theorem \ref{thm:rate_doubling1} says that, if $f$ is compactly supported in $\Omega$ then
%\begin{equation*}
%    u^\varepsilon(x) - (\nu+1)\frac{C_\alpha\varepsilon^{\alpha+1}}{d(x)^\alpha}
%\end{equation*}
%is uniformly bounded from above in $\Omega$ for all $p\in (1,2)$. Recall that $\nu = 1+C_2\delta_{0,\Omega}$ from Lemma \ref{lem:super_refined}, we can always choose $\delta_0$ small so that roughly $\nu\approx 2^+$.
%(\textcolor{red}{This sounds really weird.})
%\end{rem}

%\textcolor{orange}{(I think Theorem \ref{thm:rate_doubling1} only says $u^\epsilon-\frac{3C_\alpha \epsilon^{\alpha+1}}{d(x)^\alpha}$ is uniformly bounded above by some constant $C$. It doesn't say $u^\epsilon-\frac{3C_\alpha \epsilon^{\alpha+1}}{d(x)^\alpha}$ is uniformly bounded below. )}

\begin{rem} For general nonnegative Lipschitz data $f\in \mathrm{C}(\overline{\Omega})$, it is natural to try a cutoff function argument. Let $\chi_{{\kappa}}\in \mathrm{C}_c^\infty(\Omega)$ such that $0\leq \chi_{\kappa}\leq 1$, $\chi_\kappa = 1$ in $\Omega_{2\kappa}$ and $\mathrm{supp}\;\chi_\kappa\subset\Omega_\kappa$. Let $u^\varepsilon_\kappa\in \rmC^2(\Omega)\cap\rmC(\overline{\Omega})$ solve \eqref{eq:PDEeps} with data $f\chi_{\kappa}$. Then $u^\varepsilon_\kappa\to u^\varepsilon$ as $\kappa\to 0$ (since $f\chi_\kappa\to f$ in the weak$^*$ topology of $L^\infty(\Omega)$ and we have the continuity of the solution to \eqref{eq:PDEeps} with respect to data in this topology \cite[Remark II.1]{Lasry1989}). However, it is not clear at the moment how to quantify this rate of convergence, since $f\chi_\kappa$ does not converge to $f$ in the uniform norm, unless $f = 0$ on $\partial\Omega$.
\end{rem}

\subsection{A rate for nonnegative zero boundary data}

We prove the rate of convergence for the case where $f$ is nonnegative with $f=0$ on $\partial \Omega$.

\begin{proof}[Proof of Theorem \ref{main_thm1}]
Let $L = \Vert D f\Vert_{L^\infty(\Omega)}$ be the Lipschitz constant of $f$. For $\kappa>0$ small such that $0<\kappa<\delta_0$ and $x\in \Omega\backslash \Omega_\kappa$, let $x_0$ be the projection of $x$ onto $\partial\Omega$. We observe that
\begin{equation}\label{e:bound_aa}
    f(x) = f(x) - f(x_0) \leq L|x-x_0| = L\kappa.
\end{equation}
Define
\begin{equation*}
    g_\kappa(x) =
    \begin{cases}
        0           &\qquad\text{if}\;0\leq d(x) \leq \kappa/2,\\
        2L\left(d(x)-\kappa/2\right)       &\qquad\text{if}\;\kappa/2\leq d(x) \leq \kappa.
    \end{cases}
\end{equation*}

\noindent
It is clear that for $x\in \partial\Omega_\kappa$, $g_\kappa(x) = L\kappa \geq f(x)$ since \eqref{e:bound_aa}. Therefore, we can define the following continuous function
\begin{equation}\label{def_f_kappa}
    f_\kappa(x) =
    \begin{cases}
        0             &\qquad\text{if}\;0\leq d(x) \leq \kappa/2,\\
        \min \left\lbrace g_\kappa(x), f(x) \right\rbrace &\qquad\text{if}\;\kappa/2\leq d(x) \leq \kappa,\\
        f(x) &\qquad\text{if}\;\kappa \leq d(x).
    \end{cases}
\end{equation}
\noindent
A graph of $f_\kappa$ is given in Figure \ref{fig:f_kappa}.
\begin{figure}[h]
    \centering
    \includegraphics[scale=0.35]{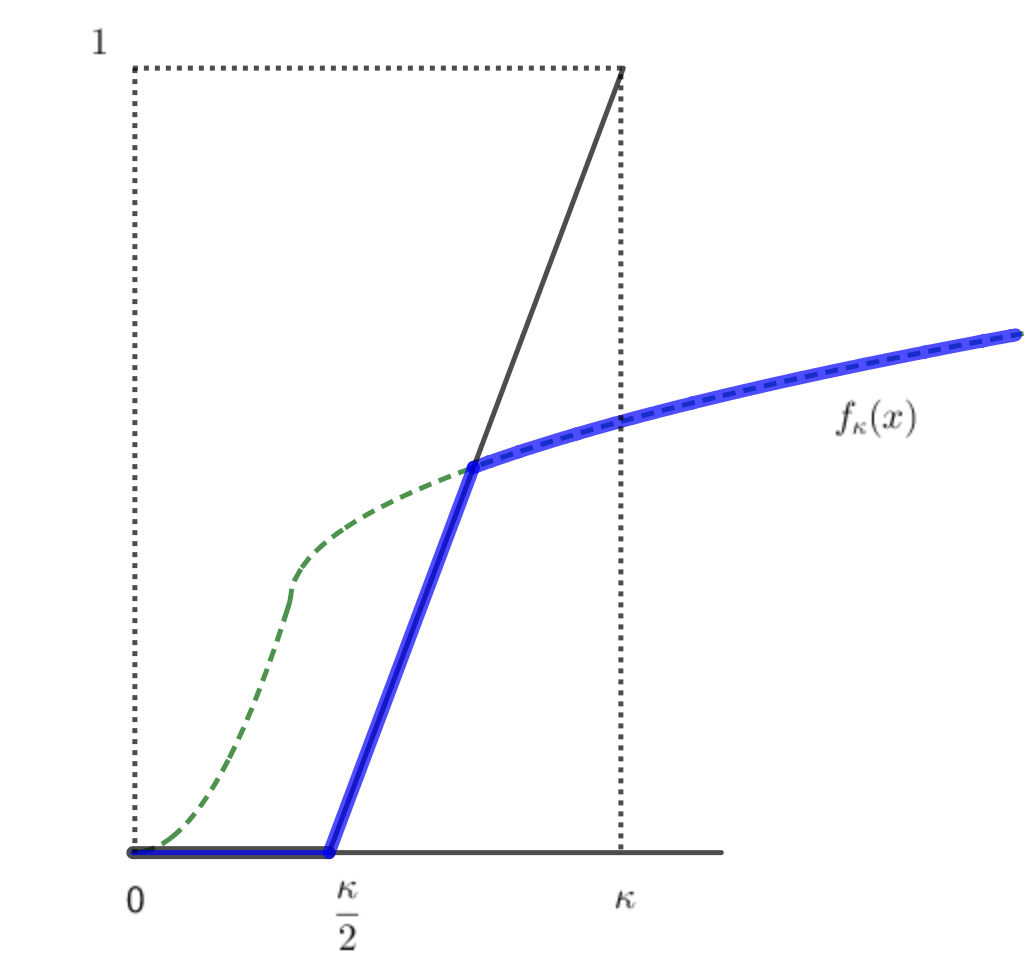}
    \caption{Graph of the function $f_\kappa$.}
    \label{fig:f_kappa}
\end{figure}
\FloatBarrier

\noindent
The continuity at $x\in \partial\Omega_\kappa$ comes from the fact that when $d(x) =\kappa$, we have $g_k(x) = L\kappa \geq f(x)$ by \eqref{e:bound_aa}. It is clear that $f_\kappa$ is Lipschitz with $\Vert f_\kappa\Vert_{L^\infty(\Omega)}\leq L$ as well and $f_\kappa\to f$ uniformly as $\kappa\to 0$. Indeed, we have $0\leq f_\kappa \leq f$ and
\begin{equation*}
    0\leq \max_{x\in \overline{\Omega}} (f(x) - f_\kappa(x)) \leq \max_{x\in \overline{\Omega}\backslash \overline{\Omega}_\kappa} (f(x) - f_\kappa(x)) = \max_{x\in \overline{\Omega}\backslash \overline{\Omega}_\kappa} f(x) \leq L\kappa.
\end{equation*}
Let $u^\varepsilon_\kappa\in \rmC^2(\Omega)\cap\rmC(\overline{\Omega})$ be the solution to \eqref{eq:PDEeps} with data $f\chi_{\kappa}$ and $u_k\in \mathrm{C}(\overline{\Omega})$ be the corresponding solution to \eqref{eq:PDE0} with data $f{\chi_\kappa}$. By the comparison principle (\cite[Corollary II.1]{Lasry1989}), we have
\begin{equation}\label{e:bound_2aa}
    0\leq u^\varepsilon(x) - u^\varepsilon_\kappa(x) \leq L\kappa \qquad\text{for}\;x\in \Omega.
\end{equation}
By the comparison principle for \eqref{eq:PDE0}, we also have
\begin{equation}\label{e:bound_2ab}
    0\leq u(x) - u_\kappa(x) \leq L\kappa \qquad\text{for}\;x\in \Omega.
\end{equation}
If $1<p<2$, by Theorem \ref{thm:rate_doubling1}, there exists a constant $C$ independent of $\kappa$ such that
\begin{equation}\label{e:bound_2ac}
    -C\sqrt{\varepsilon}\leq u^\varepsilon_\kappa(x) - u_\kappa(x)\leq C\left[\sqrt{\varepsilon} + \left(\frac{\varepsilon}{\kappa}\right)^{\alpha+2} + \frac{\varepsilon^{\alpha+1}}{d(x)^\alpha}\right], \qquad x\in \Omega.
\end{equation}
Combining \eqref{e:bound_2aa}, \eqref{e:bound_2ab} and \eqref{e:bound_2ac}, we obtain
\begin{equation*}
\begin{split}
   -C\sqrt{\varepsilon}\leq u^\varepsilon(x) - u(x) &= \Big(u^\varepsilon(x) - u^\varepsilon_\kappa(x)\Big) + \Big(u^\varepsilon_\kappa(x) - u_\kappa(x)\Big) + \Big(u_\kappa(x) - u(x)\Big) \\
    &\leq L\kappa + C\left[\sqrt{\varepsilon}  + \left(\frac{\varepsilon}{\kappa}\right)^{\alpha+2} + \frac{\varepsilon^{\alpha+1}}{d(x)^\alpha}\right], \qquad x\in \Omega.
\end{split}
\end{equation*}
Choose $\kappa = \sqrt{\varepsilon}$ and we deduce that
\begin{equation*}
    -C\sqrt{\varepsilon}\leq u^\varepsilon(x) - u(x) \leq C\sqrt{\varepsilon} + \frac{C\varepsilon^{\alpha+1}}{d(x)^\alpha}
\end{equation*}
for $x\in \Omega$. Thus, the conclusion follows. \\

\noindent If $p=2$, by Theorem \ref{thm:rate_doubling1}, there exists a constant $C$ independent of $\kappa$ such that
\begin{equation}\label{e:bound_2acp=2}
    -C\sqrt{\varepsilon}\leq u^\varepsilon_\kappa(x) - u_\kappa(x) \leq C\left[\sqrt{\varepsilon} + \left(\frac{\varepsilon}{\kappa}\right)^2 + \varepsilon\log\left(\frac{1}{d(x)}\right)\right], \qquad x\in \Omega.
\end{equation}
Combining \eqref{e:bound_2aa}, \eqref{e:bound_2ab} and \eqref{e:bound_2acp=2}, we obtain
\begin{equation*}
\begin{split}
   -C\sqrt{\varepsilon}\leq u^\varepsilon(x) - u(x) &= \Big(u^\varepsilon(x) - u^\varepsilon_\kappa(x)\Big) + \Big(u^\varepsilon_\kappa(x) - u_\kappa(x)\Big) + \Big(u_\kappa(x) - u(x)\Big) \\
    &\leq L\kappa + C\left[\sqrt{\varepsilon} +\;    \left(\frac{\varepsilon}{\kappa}\right)^{2} + \varepsilon\log\left(\frac{1}{d(x)}\right) \right], \qquad x\in \Omega.
\end{split}
\end{equation*}
Choose $\kappa = \varepsilon$ and we deduce that
\begin{equation*}
    -C\sqrt{\varepsilon}\leq u^\varepsilon(x) - u(x) \leq C\sqrt{\varepsilon} +\varepsilon\log\left(\frac{1}{d(x)}\right)
\end{equation*}
for $x\in \Omega$. Thus, the conclusion follows.
\end{proof}

\section{Improved one-sided rate of convergence}
In this section, we assume $f\in \rmC^2(\overline{\Omega})$ (or uniformly semiconcave in $\overline{\Omega}$) such that $f = 0$ on $\partial\Omega$ and $f\geq 0$. It is known that for the problem on $\R^n$, namely,
\begin{equation*}
    u(x) + |Du|^p - f(x) = 0 \qquad\text{in}\;\R^n,
\end{equation*}
if $f$ is semiconcave in the whole space $\R^n$, then the solution $u$ is also semiconcave (Theorem \ref{convex}, see also \cite{Calder2021}).

\begin{rem}\label{rem:heuristic} The heuristic idea that we will use in this section is the following. Assume that $u^\varepsilon(x) - u(x)$ has a maximum over $\overline{\Omega}$ at some interior point $x_0\in \Omega$. Then by the equation \eqref{eq:PDEeps} at $x_0$ and the supersolution test for \eqref{eq:PDE0} at $x_0$, we obtain
\begin{equation*}
    \max_{x\in \overline{\Omega}}\Big( u^\varepsilon(x) - u(x)\Big) \leq u^\varepsilon(x_0) - u(x_0) \leq \varepsilon \Delta u^\varepsilon(x_0).
\end{equation*}
If $u$ is uniformly semiconcave in $\overline{\Omega}$, then $\Delta u^\varepsilon(x_0)\leq \Delta u(x_0) \leq C$. Thus, we obtain a better one-sided rate $\mathcal{O}(\varepsilon)$ for $u^\varepsilon-u$. However, there are a couple of problems with this argument. Firstly, as $u^\varepsilon = +\infty$ on $\partial\Omega$, we need to subtract an appropriate term from $u^\varepsilon$ to make a maximum over $\overline{\Omega}$ happen in the interior. Secondly, unless $f\in \mathrm{C}^2_c(\Omega)$, in general, $u$ is not uniformly semiconcave but only \emph{locally semiconcave}. In this section, we provide estimates on the local semiconcavity constant of $u$ and rigorously show how the upper bound of $u^\varepsilon-u$ can be obtained.
\end{rem}

From Lemma \ref{lem:f=0}, we have $u = 0$ on $\partial\Omega$. It is clear that the solution $u$ to \eqref{eq:PDE0} is also the unique solution to the following Dirichlet boundary problem
\begin{equation}\label{eq:Dirichlet}
    \begin{cases}
        u(x) + |Du(x)|^p = f(x) &\qquad\text{in}\;\Omega,\\
        \quad\quad \quad\quad\;\;\; u(x) = 0 &\qquad\text{on}\;\partial\Omega.
    \end{cases}
\end{equation}
Since $H(x,\xi) = |\xi|^p - f(x)$, the corresponding Legendre transform is
\begin{equation*}
    L(x,v) = C_p|v|^q + f(x)
\end{equation*}
where $p^{-1} + q^{-1} = 1$ and $C_p$ is defined in Lemma \ref{lem:f=0}. Let us extend $f$ to a function $\tilde{f}:\R^n\to \R$ by setting $\tilde{f}(x) = 0$ for $x\notin \Omega$.
\begin{defn} Define
\begin{equation*}
    C^k_0(\overline{\Omega}) = \Big\{\varphi\in \mathrm{C}^k(\overline{\Omega}):  D^\beta\varphi(x) = 0\;\text{on}\;\partial\Omega\;\text{for all}\;0\leq |\beta|\leq k\Big\}.
\end{equation*}
\end{defn}
We summarize the results about the semiconcavity of $u$ as follows.
\begin{thm}[Semiconcavity]\label{thm:newsemi} Assume $f \geq 0$, $f = 0$ on $\partial\Omega$ and $f$ is uniformly semiconcave in $\overline{\Omega}$ with semiconcavity constant $c$. Let $u$ be the solution to \eqref{eq:PDE0}.
\begin{itemize}
    \item[(i)] If $\tilde{f}$ is uniformly semiconcave in $\R^n$, then $u$ is uniformly semiconcave in $\overline{\Omega}$. %This condition holds for $\mathrm{C}_c^2(\Omega)$ and  $\mathrm{C}_0^2(\Omega)$.
    \item[(ii)] In general, $u$ is locally semiconcave. More specifically, there exists a constant $C > 0$ independent of $x \in \Omega$ such that $\forall x \in \Omega$,
\begin{equation}
    u(x+h)-2u(x)+u(x-h)\leq \frac{C}{d(x)}|h|^2,
\end{equation}
$\forall h \in \mathbb{R}^n$ with $|h| \leq M_x$ for some constant $M_x$ that depends on $x$.
\end{itemize}
\end{thm}
The proof of Theorem \ref{thm:newsemi} is given in Appendix.

\begin{rem} If $f\in \mathrm{C}_c^2(\R^n)$ (or $\mathrm{C}_0^2(\overline{\Omega})$), then $f$ is uniformly semiconcave with semiconcavity constant
\begin{equation}\label{def_c}
    c = \max \big\lbrace  D_{\xi\xi}f(x): |\xi|=1, x\in \R^n \big\rbrace\geq 0.
\end{equation}
Also, the condition that $\tilde{f}$ is semiconcave in $\R^n$ holds for $\mathrm{C}_c^2(\Omega)$ and $\mathrm{C}_0^2(\overline{\Omega})$. %We also note that the regularity of $f$ can be relaxed to that $f$ is uniformly semiconcave in the interior with a linear modulus $c>0$, i.e., %\begin{equation*}
%    f(x+h)-2f(x)+f(x-h)\leq c|h|^2, \quad \forall x, h \in \mathbb{R}^n \,\, \text{such that} \, \, x+h, x, \text{and}\, \, x-h \in \overline{\Omega}.
%\end{equation*}
\end{rem}

The following lemma is a refined version of the local gradient bound in Theorem \ref{thm:grad_1}. We follow \cite[Theorem 3.1]{Armstrong2015a} where the authors use Bernstein's method inside a doubling variable argument and explicitly keep track of all the dependencies. We refer the reader to \cite{barles_weak_1991,capuzzo_dolcetta_holder_2010} and the references therein for related versions of the gradient bound. We believe this result is new in the literature since it is uniform in $\varepsilon$, namely, we give the explicit dependence of the gradient bound on $d(x)$. It also indicates that the boundary layer is a strip of size $\mathcal{O}(\varepsilon)$ from the boundary.

\begin{lem}\label{lem:boundDu^eps} For all $\varepsilon$ small enough, there exists a constant $C$ independent of $\varepsilon$ such that
\begin{equation}\label{eq:es_final}
|Du^\varepsilon(x)| \leq C\left( 1 + \left(\frac{\varepsilon}{d(x)}\right)^{\alpha+1}\right) \qquad\text{for}\;x\in \Omega.
\end{equation}
%As a consequence, there exists a constant $C$ independent of $\varepsilon$ such that
%\begin{equation*}
%    |\Delta u^\varepsilon(x)| \leq C\left(1+\left(\frac{\varepsilon}{d(x)}\right)^{\alpha+2}\right) \qquad\text{for}\;x\in \Omega.
%0\end{equation*}
\end{lem}

\begin{proof}[Proof of Lemma \ref{lem:boundDu^eps}] %Let us consider $0<\varepsilon < \min\{4\delta_0,1\}$.
Fix $x_0 \in \Omega\backslash \Omega_{\delta_0}$. Let $\delta := \frac{1}{4}d(x_0)$ and
\begin{equation*}
    v(x) := \frac{1}{\delta}u^\varepsilon(x_0+\delta x), \qquad x\in B(0,2).
\end{equation*}
Then $v$ solves
\begin{equation}\label{eq:eqn_of_v}
    \delta v(x) + |Dv(x)|^p - \tilde{f}(x) - \frac{\varepsilon}{\delta}\Delta v(x) = 0 \qquad\text{in}\;B(0,2),
\end{equation}
where $\tilde{f}(x): = f(x_0+\delta x)$ on $\overline{B(0,2)}$. Note that $\Vert \tilde{f}\Vert_{L^\infty}\leq \Vert f\Vert_{L^\infty}$ and
\begin{equation*}
    B(x_0,2\delta)\subset \Omega_{2\delta}\subset\subset\Omega.
\end{equation*}
By Lemma \ref{lem:super_refined}, there is a constant $C$ independent of $\delta,\varepsilon$ such that
\begin{equation*}
     \delta\Vert v\Vert_{L^\infty\left(B\left(0,\frac{3}{2}\right)\right)} \leq \Vert u^\varepsilon\Vert_{L^\infty(\Omega_{2\delta}))} \leq C\left(1 + \frac{\varepsilon^{\alpha+1}}{\delta^\alpha}\right).
\end{equation*}
Apply Theorem 3.1 in \cite{Armstrong2015a} to obtain
\begin{equation*}
\begin{split}
    \sup_{x\in B(0,1)}|Dv(x)| &\leq C\left[\left(\frac{\varepsilon}{\delta}\right)^{\frac{1}{p-1}} + \left(\Vert f\Vert_{L^\infty}+\delta \Vert v\Vert_{L^\infty\left(B\left(0,\frac{3}{2}\right)\right)}\right)^{\frac{1}{p}}\right] \\
    &\leq C\left[\left(\frac{\varepsilon}{\delta}\right)^{\alpha+1} + \left(1+\frac{\varepsilon^{\alpha+1}}{\delta^\alpha}\right)^{\frac{\alpha+1}{\alpha+2}}\right]\leq C \left(1+\left(\frac{\varepsilon}{\delta}\right)^{\alpha+1}\right),
    %&\leq C\left[\left(\frac{\varepsilon}{\delta}\right)^{\alpha+1} + \left(1+\left(\frac{\alpha+1}{\alpha+2}\right)\frac{\varepsilon^{\alpha+1}}{\delta^\alpha}\right)\right] \leq C\left[1+2\left(\frac{\varepsilon}{\delta}\right)^{\alpha+1} \right]
\end{split}
\end{equation*}
where $p = \frac{\alpha+2}{\alpha+1}$ and $\alpha+1 = \frac{1}{p-1}$.
%where we use Bernoulli's inequality and the fact that $\delta \leq 1$.
Plugging in $\delta = \frac{1}{4}d(x_0)$, we obtain
\begin{equation*}
    |Du^\varepsilon(x_0)| = |Dv(0)| \leq C\left(1+ \left(\frac{\varepsilon}{d(x_0)}\right)^{\alpha+1}\right).
\end{equation*}
In other words, we have \eqref{eq:es_final} for all $x\in \Omega\backslash\Omega_{\delta_0}$. On the other hand, from Theorem \ref{thm:grad_1}, there exists a constant $C$ independent of $\varepsilon$ such that $|Du^\varepsilon(x)|\leq C$ for all $x\in \Omega_{\delta_0}$. Thus, the proof is complete.
\end{proof}

\begin{proof}[Proof of Theorem \ref{thm:rate_doubling2}] For $1<p<2$, we proceed as in the proof of Theorem \ref{thm:rate_doubling1} to obtain
\begin{align}
    &0\leq u^\varepsilon(x) \leq \tilde{u}^\varepsilon(x)  \leq \frac{\nu C_\alpha\varepsilon^{\alpha+1}}{d_\kappa(x)^{\alpha}} + C_3\left(\frac{4\varepsilon}{\kappa}\right)^{\alpha+2} \label{annulus2a}
    %&0\leq u^\varepsilon(x)\leq \tilde{u}^\varepsilon(x) \leq \nu \varepsilon \log\left(\frac{1}{d_\kappa(x)}\right)+C_\nu\left(\frac{4\varepsilon}{\kappa}\right)^{2}\qquad\text{for}\;p=2,\label{annulus2p=2a}
\end{align}
for $x\in U_\kappa$. Let
\begin{equation*}
    \psi^\varepsilon(x) := u^\varepsilon(x) - \frac{\nu C_\alpha \varepsilon^{\alpha+1}}{d(x)^\alpha}, \qquad x\in \Omega,
\end{equation*}
where $\nu > 1$ is chosen as in Lemma \ref{lem:super_refined}. It is clear that $u-\psi^\varepsilon$ has a local minimum at some point $x_0\in \Omega$ since $\psi^\varepsilon(x)\to -\infty$ as $x\to \partial\Omega$. The normal derivative test gives us
\begin{equation*}
    D\psi^\varepsilon(x_0) = Du^\varepsilon(x_0) + \nu C_\alpha\alpha \left(\frac{\varepsilon}{d(x_0)}\right)^{\alpha+1} D d(x_0) \in D^-u(x_0).
\end{equation*}
There are two cases to consider:
\begin{itemize}
    \item If $\displaystyle d(x_0)< \frac{1}{2}\kappa$, then as in the proof of Theorem \ref{thm:rate_doubling1}, $x_0\in U_\kappa$ and $d_\kappa(x_0) = d(x_0)$. By the definition of $x_0$, for any $x\in \Omega$, there holds
    \begin{align*}
        u(x) - \left(u^\varepsilon(x) - \frac{\nu C_\alpha \varepsilon^{\alpha+1}}{d(x)^\alpha}\right) \geq u(x_0) - \left(u^\varepsilon(x_0) - \frac{\nu C_\alpha \varepsilon^{\alpha+1}}{d(x_0)^\alpha}\right).
    \end{align*}
    Therefore,
    \begin{equation*}
    \begin{split}
        u^\varepsilon(x) - u(x) - \frac{\nu C_\alpha \varepsilon^{\alpha+1}}{d(x)^\alpha} &\leq  \left(u^\varepsilon(x_0) - \frac{\nu C_\alpha \varepsilon^{\alpha+1}}{d(x_0)^\alpha}\right) - u(x_0) \leq C_3 \left(\frac{4\varepsilon}{\kappa}\right)^{\alpha+2}
    \end{split}
    \end{equation*}
    thanks to \eqref{annulus2a}. Thus, in this case
    \begin{equation*}
        u^\varepsilon(x) - u(x) \leq \frac{\nu C_\alpha \varepsilon^{\alpha+1}}{d(x)^\alpha} + C_3 \left(\frac{4\varepsilon}{\kappa}\right)^{\alpha+2}, \qquad x\in \Omega.
    \end{equation*}
    \item If $\displaystyle d(x_0) \geq \frac{1}{2}\kappa$, from the fact that $u$ is semiconcave in $\Omega$ with a linear modulus $c(x)$ as in Theorem \ref{thm:newsemi}, we have
\begin{equation*}
    D^2\psi^\varepsilon(x_0) \prec c(x_0)\;\mathbb{I}_n,
\end{equation*}
which implies that
\begin{equation}\label{mybound}
    \Delta \psi^\varepsilon(x_0) \leq nc(x_0) \leq \frac{Cn}{d(x_0)} \leq \frac{Cn}{\kappa}.
\end{equation}
%\textcolor{red}{(Is $c_0$ the constant $C$ in Theorem 4.1?) -- Yes!}

In other words, we have
\begin{equation*}
    \varepsilon\Delta u^\varepsilon (x_0) - \frac{\nu C_\alpha \alpha(\alpha+1)\varepsilon^{\alpha+2}}{d(x_0)^{\alpha+2}}| Dd(x_0)|^2 + \frac{\nu C_\alpha \alpha \varepsilon^{\alpha+2}}{d(x_0)^{\alpha+1}}\Delta d(x_0) \leq \frac{Cn\varepsilon}{\kappa}.
\end{equation*}
Since $d(x_0)\geq \frac{1}{2}\kappa$, we can further deduce that
\begin{equation}\label{delta_psi}
    \varepsilon\Delta u^\varepsilon(x_0) \leq \frac{Cn\varepsilon}{\kappa} + \frac{C\varepsilon^{\alpha+2}}{d(x_0)^{\alpha+2}} \leq \frac{Cn\varepsilon}{\kappa} +  C\left(\frac{\varepsilon}{\kappa}\right)^{\alpha+2},
\end{equation}
where $C$ is independent of $\varepsilon$. Since $\psi^\varepsilon\in \rmC^2(\Omega)$, the viscosity supersolution test for $u$ gives us
\begin{equation}\label{convex1}
    u(x_0) + \left|Du^\varepsilon(x_0) + \frac{\nu C_\alpha \alpha \varepsilon^{\alpha+1}}{d(x_0)^{\alpha+1}}D d(x_0)\right|^p - f(x_0) \geq 0.
\end{equation}
On the other hand, since $u^\varepsilon$ solves \eqref{eq:PDEeps}, we have
\begin{equation}\label{convex2}
    u^\varepsilon(x_0) + |Du^\varepsilon(x_0)|^p - f(x_0) - \varepsilon \Delta u^\varepsilon(x_0) = 0.
\end{equation}
Combine \eqref{convex1} and \eqref{convex2} to obtain that
\begin{align}\label{convex3}
    u^\varepsilon(x_0) - u(x_0) &\leq \left|Du^\varepsilon(x_0) + \frac{\nu C_\alpha \alpha \varepsilon^{\alpha+1}}{d(x_0)^{\alpha+1}}D d(x_0)\right|^p - |Du^\varepsilon(x_0)|^p + \varepsilon\Delta u^\varepsilon(x_0).
\end{align}
By Lemma \eqref{lem:boundDu^eps}, we can bound $Du^\varepsilon(x_0)$ as
\begin{equation}\label{eq:bound_Du^eps}
    |Du^\varepsilon(x_0)| \leq C + C \left(\frac{\varepsilon}{d(x_0)}\right)^{\alpha+1} \leq C +C\left(\frac{\varepsilon}{\kappa}\right)^{\alpha+1}
\end{equation}
since $d(x_0)\geq \frac{1}{2}\kappa$. We estimate the gradient terms on the right hand side of \eqref{convex3} using \eqref{eq:bound_Du^eps} as follows.
\begin{align}\label{convex4}
    &\left|Du^\varepsilon(x_0) + \frac{\nu C_\alpha \alpha \varepsilon^{\alpha+1}}{d(x_0)^{\alpha+1}}D d(x_0)\right|^p - |Du^\varepsilon(x_0)|^p\nonumber \\
    &\qquad \qquad \leq p\left(|Du^\varepsilon(x_0)| + \frac{\nu C_\alpha \alpha \varepsilon^{\alpha+1}}{d(x_0)^{\alpha+1}}\left|D d(x_0)\right|\right)^{p-1} \frac{\nu C_\alpha \alpha \varepsilon^{\alpha+1}}{d(x_0)^{\alpha+1}}|D d(x_0)|\nonumber\\
    &\qquad \qquad \leq p\left(C + C\left(\frac{\varepsilon}{\kappa}\right)^{\alpha+1}\right)^{p-1} C\left(\frac{\varepsilon}{\kappa}\right)^{\alpha+1}\leq C\left(\frac{\varepsilon}{\kappa}\right)^{\alpha+1}\left(1 + \left(\frac{\varepsilon}{\kappa}\right)\right),
\end{align}
where $C$ is a constant depending only on $\nu$, $\alpha$, and $d$. Plugging \eqref{delta_psi} and \eqref{convex4} in the right hand side of \eqref{convex3}, we get
\begin{equation*}
    u^\varepsilon(x_0) - u(x_0) \leq \frac{Cn\varepsilon}{\kappa} + C\left(\frac{\varepsilon}{\kappa}\right)^{\alpha+1}\left(1 + \left(\frac{\varepsilon}{\kappa}\right)\right).
\end{equation*}
Therefore,
\begin{equation*}
    u^\varepsilon(x) - u(x) \leq \frac{\nu C_\alpha \varepsilon^{\alpha+1}}{d(x)^\alpha} +C\left( \left(\frac{\varepsilon}{\kappa}\right)^{\alpha+1} + \left(\frac{\varepsilon}{\kappa}\right)^{\alpha+2}\right) + \frac{Cn\varepsilon}{\kappa}, \qquad x\in \Omega.
\end{equation*}
\end{itemize}
%From the two cases, the conclusion follows.

For $p = 2$, the argument is similar. We take $\psi^\varepsilon(x):=u^\varepsilon(x)-\nu \varepsilon \log \left(\frac{1}{d(x)}\right)$ instead and still $u-\psi^\varepsilon$ attains a local minimum at some point $x_0 \in \Omega$. Carrying out the similar computations as in the case of $1 < p < 2$, we have:
\begin{itemize}
    \item If $\displaystyle d(x_0) < \frac{1}{2} \kappa$, then
    \begin{equation*}
        u^\varepsilon(x)-u(x) \leq \nu \varepsilon \log \left( \frac{1}{d(x)} \right) + C\left( \frac{4\varepsilon}{\kappa}\right)^2, \quad x \in \Omega.
    \end{equation*}
    \item If $\displaystyle d(x_0) \geq \frac{1}{2} \kappa$, then
    \begin{equation*}
        u^\varepsilon(x)-u(x) \leq \nu \varepsilon \log\left(\frac{1}{d(x)}\right) + C\left(\left( \frac{\varepsilon}{\kappa} \right)+  \left( \frac{\varepsilon}{\kappa} \right)^2\right) + \frac{Cn\varepsilon}{\kappa}, \quad x \in \Omega.
    \end{equation*}
\end{itemize}
From these two cases, the conclusion for $p=2$ follows.
\end{proof}
\begin{rem}\label{rem:nice} If $f\in \mathrm{C}^2(\overline{\Omega})$ with $f = 0,Df = 0$ and $D^2f = 0$ on $\partial\Omega$, then \eqref{mybound} can be improved to $\Delta \psi^\varepsilon(x_0) \leq nc$ where $c$ is the semiconcavity constant of $f$, and thus the final estimate becomes
\begin{equation*}
    u^\varepsilon(x) - u(x) \leq \frac{\nu C_\alpha \varepsilon^{\alpha+1}}{d(x)^\alpha} +C\left( \left(\frac{\varepsilon}{\kappa}\right)^{\alpha+1} + \left(\frac{\varepsilon}{\kappa}\right)^{\alpha+2}\right) + nc\varepsilon, \qquad x\in \Omega.
\end{equation*}
\end{rem}

\begin{rem} \quad
\begin{itemize}
    \item We only need the local gradient bound in Theorem \ref{thm:grad_1} to obtain the local rate of convergence $\mathcal{O}(\varepsilon)$ in \eqref{convex4}. However, to make the dependence on $\kappa$ explicit, we need to bound $Du^\varepsilon(x_0)$ as in \eqref{convex4}.
    \item Another way to get \eqref{eq:bound_Du^eps} without using Lemma \ref{lem:boundDu^eps} (which is true for all $x\in \Omega$) is using the fact that $D\psi^\varepsilon(x_0)\in D^-u(x_0)$, which implies
    \begin{equation*}
        |D\psi^\varepsilon(x_0)| = \left|Du^\varepsilon(x_0) + \nu C_\alpha\alpha \left(\frac{\varepsilon}{d(x_0)}\right)^{\alpha+1} D d(x_0) \right| \leq C_0
    \end{equation*}
    since $u$ is Lipschitz with constant $C_0$.
\end{itemize}
\end{rem}

Before giving a proof for Corollary \ref{cor:key}, we need to modify the construction of the cutoff function in the proof of Theorem \ref{main_thm1}.
\begin{lem}\label{conca} Assume $f\in \mathrm{C}^2(\overline{\Omega})$ such that $f=0$ and $Df = 0$ on $\partial\Omega$. For all $\kappa>0$ small enough, there exists $f_\kappa\in \mathrm{C}^2_c(\Omega)$ such that
\begin{equation*}
    \Vert f_\kappa - f\Vert_{L^\infty(\Omega)}\leq C\kappa^2 \qquad\text{and}\qquad \Vert D^2f_\kappa\Vert_{L^\infty(\Omega)}\leq C
\end{equation*}
where $C$ is independent of $\kappa$.
\end{lem}
\begin{proof} Choose a smooth function $\chi \in \mathrm{C}^\infty(\R)$ such that $\chi\geq 0$, $\chi = 0$ if $x\leq 1$, $\chi = 1$ if $x\geq 2$ and $|\chi'|\leq 2$ in $\R$. The graph of $\chi$ is shown as in Figure \ref{fig:chi}.

\begin{figure}[ht]
    \centering
    %\incfig{Domains}
    \includegraphics[scale = 0.45]{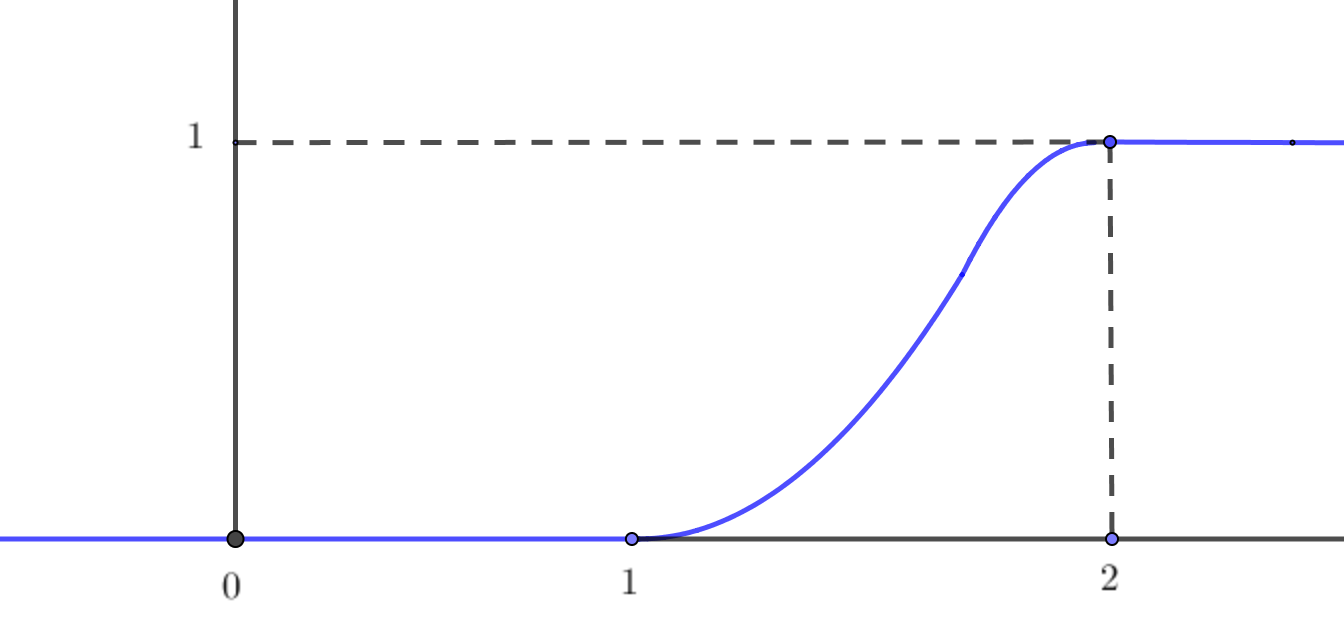}
    \caption{The graph of $\chi(x)$.}
    \label{fig:chi}
\end{figure}
\FloatBarrier
For $\kappa>0$ such that $0<2\kappa<\delta_0$ and $x\in \Omega\backslash \Omega_{2\kappa}$, let $x_0$ be the projection of $x$ onto $\partial\Omega$ and denote by $\nu(x_0)$ the outward unit normal vector at $x_0$. Write $x = x_0 - d(x)\nu(x_0)$ where $d(x)\leq 2\kappa$. We have
\begin{equation*}
    f(x) = f(x_0)-Df(x_0)\cdot \nu(x_0) d(x) + \int_0^{d(x)}(d(x)-s)\nu(x_0)\cdot D^2f(x_0-s\nu(x_0))\cdot \nu(x_0)ds.
\end{equation*}
Since $f = 0$ and $Df = 0$ on $\partial\Omega$, we deduce that
\begin{equation}\label{bound_near_bdr}
    |f(x)| \leq  \left( \left\Vert\frac{1}{2} D^2f \right\Vert_{L^\infty(\overline{\Omega})}\right)d(x)^2 \leq C\kappa^2 \qquad\text{and}\qquad
    |Df(x)|\leq C\kappa
\end{equation}
for all $d(x)\leq 2\kappa$. %Let us define
%\begin{equation*}
%    g_\kappa(x) = \chi\left(\frac{d(x)}{\kappa}\right) \qquad\text{for}\;x\in \overline{\Omega}.
%\end{equation*}
%It is clear that $g_\kappa$ is twice continuously differentiable for $0\leq d(x)<\delta_0$ and also
%\begin{equation*}
%\begin{cases}
%    g_\kappa(x) = 0 \qquad\text{if}\;d(x)<\kappa,\\
%    g_\kappa(x) =1  \qquad\text{if}\;d(x)\geq 2\kappa.
%\end{cases}
%\end{equation*}
Define
\begin{equation*}
    f_\kappa(x) = f(x)\chi\left(\frac{d(x)}{\kappa}\right) \qquad\text{for}\;x\in \overline{\Omega}.
\end{equation*}
It is clear that $0\leq f_\kappa(x)\leq f(x)$ for all $x\in \overline{\Omega}$ and $f_\kappa(x) = f(x)$ if $d(x)\geq 2\kappa$. Furthermore, we observe that
\begin{equation*}
    0\leq \max_{x\in \overline{\Omega}} \big(f(x) - f_\kappa(x)\big) \leq \max_{0\leq d(x) \leq 2\kappa} \big(f(x) - f_\kappa(x)\big)\leq  \max_{0\leq d(x) \leq 2\kappa} f(x) \leq  C\kappa^2.
\end{equation*}
We have
\begin{equation*}
    Df_\kappa(x) = Df(x)\chi\left(\frac{d(x)}{\kappa}\right)+ f(x)\chi'\left(\frac{d(x)}{\kappa}\right)\frac{Dd(x)}{\kappa}
\end{equation*}
and
\begin{equation*}
\begin{split}
    D^2f_\kappa(x) = D^2f(x)\chi\left(\frac{d(x)}{\kappa}\right) &+ 2\chi'\left(\frac{d(x)}{\kappa}\right)\frac{Df(x)\otimes Dd(x)}{\kappa} \\
    &+ f(x)\left(\chi''\left(\frac{d(x)}{\kappa}\right)\frac{Dd(x)\otimes Dd(x)}{\kappa^2} + \chi'\left(\frac{d(x)}{\kappa}\right)\frac{D^2d(x)}{\kappa}\right)
\end{split}
\end{equation*}
is uniformly bounded thanks to \eqref{bound_near_bdr}.
\end{proof}

\begin{proof}[Proof of Corollary \ref{cor:key}] Let $u^\varepsilon_\kappa\in \rmC^2(\Omega)\cap \rmC(\overline{\Omega})$ be the solution to \eqref{eq:PDEeps} and $u_k$ be the solution to \eqref{eq:PDE0} with $f$ replaced by $f_k$, respectively. It is clear that
\begin{equation*}
  0\leq u^\varepsilon(x) - u^\varepsilon_\kappa(x)\leq C\kappa ^2 \qquad\text{for}\;x\in \Omega
\end{equation*}
and
\begin{equation*}
    0\leq u(x) - u_\kappa(x)\leq C\kappa ^2  \qquad\text{for}\;x\in \Omega.
\end{equation*}
Therefore,
\begin{equation}\label{cool}
    u^\varepsilon(x)-u(x)\leq 2C\kappa^2 + \Big(u^\varepsilon_\kappa(x) - u_\kappa(x)\Big).
\end{equation}
By Theorem \ref{thm:rate_doubling2} and Remark \ref{rem:nice}, as $f_\kappa\in \mathrm{C}_c^2(\Omega)$ with a uniform bound on $D^2f_\kappa$, we have
\begin{equation*}
\begin{aligned}
    u_\kappa^\varepsilon(x) - u_\kappa(x) &\leq\frac{\nu C_\alpha \varepsilon^{\alpha+1}}{d(x)^\alpha} + C\left(\left(\frac{\varepsilon}{\kappa}\right)^{\alpha+1} +\left(\frac{\varepsilon}{\kappa}\right)^{\alpha+2}\right) + 4nC\varepsilon,  &\quad p <2, \\
    u_\kappa^\varepsilon(x) - u_\kappa(x) & \leq \nu \varepsilon \log\left( \frac{1}{d(x)}\right)+C \left(\left(\frac{\varepsilon}{\kappa}\right)+\left(\frac{\varepsilon}{\kappa}\right)^2\right)+ 4nC\varepsilon, &\quad p=2
\end{aligned}
\end{equation*}
for some constant $C$ independent of $\kappa$. Choose $\kappa = \varepsilon^{\gamma}$ with $\gamma \in (0,1)$. Then \eqref{cool} becomes
\begin{align*}
    &u^\varepsilon(x) - u(x) \leq C\varepsilon^{2\gamma} + C\varepsilon + \frac{C\varepsilon^{\alpha+1}}{d(x)^\alpha} + C\varepsilon^{(1-\gamma)(\alpha+1)}, &\qquad p < 2,\\
    &u^\varepsilon(x) - u(x) \leq C\varepsilon^{2\gamma} + C\varepsilon + C\varepsilon |\log d(x)| + C\varepsilon^{1-\gamma}, &\qquad p = 2.
\end{align*}
% If $p = 2$, then $\gamma = 1/2$ is the best value to choose, which implies the $\mathcal{O}(\sqrt{\varepsilon})$ estimate in Theorem \ref{main_thm1}. If $p<2$, by setting $\gamma = (1-\gamma)(\alpha+1)$, we can get the best value of $\gamma$, that is,
% \begin{equation*}
%     \gamma = \frac{\alpha+1}{\alpha+2} = \frac{1}{p} > \frac{1}{2},
% \end{equation*}
% and we obtain a better estimate $\mathcal{O}(\varepsilon^{1/p})$.
If \( p = 2 \), the optimal choice of \( \gamma \) is given by \( 2\gamma = 1 - \gamma \), i.e., \( \gamma = \frac{1}{3} \), which yields a rate of \( \mathcal{O}(\varepsilon^{2/3}) \), an improvement over the \( \mathcal{O}(\sqrt{\varepsilon}) \) estimate in Theorem \ref{main_thm1}. \medskip

\noindent
If $p< 2$, by setting $2\gamma = (1-\gamma)(\alpha+1)$, we can get the best value of $\gamma$, that is, $\gamma = \frac{\alpha+1}{\alpha+3}$,
and we obtain an improved estimate of \( \mathcal{O}\left(\varepsilon^{\frac{2(\alpha+1)}{\alpha+3}}\right) \), noting that \( \frac{2(\alpha+1)}{\alpha+3} = \frac{1}{p-1/2} > \frac{1}{p} > \frac{1}{2} \).
\end{proof}

\begin{rem} If we do not assume $Df = 0$ on $\partial\Omega$, then the best we can get from the above argument is
\begin{equation*}
    u^\varepsilon(x) - u(x) \leq C\varepsilon^{\gamma} + C\varepsilon^{1-\gamma} + \frac{C\varepsilon^{\alpha+1}}{d(x)^\alpha} + C\varepsilon^{(1-\gamma)(\alpha+1)}, \qquad p < 2
\end{equation*}
and we obtain the rate $\mathcal{O}(\varepsilon^{1/2})$ again.
\end{rem}

%\begin{rem} Corollary \ref{cor:key} can be extended to the case $f$ is semiconcave in $\Omega$ with some additional assumptions, for example $f =0$ on $\partial\Omega$ and $f>0$ in $\Omega$. The key element of the approximation above is, we need a semiconcave cutoff $f_\kappa$ satisfies either $f_\kappa$ is uniformly semiconcave in $\overline{\Omega}_\kappa$ or $\mathrm{supp}(f_\kappa)$ has a $\rmC^2$ boundary so that we can apply the previous results in $U_\kappa = \mathrm{supp}(f_\kappa)$. This technical direction will be addressed in a future work.
%\end{rem}

\setcounter{section}{0}
\renewcommand\thesection{\Alph{section}}
\section{Appendix}
\subsection{Estimates on solutions}
We present here a proof for the gradient bound of the solution to \eqref{eq:PDEeps} using Bernstein's method (see also \cite{Lasry1989,lions_quelques_1985}). Another proof using Bernstein's method inside a doubling variable argument is given in \cite{Armstrong2015a}.

\begin{proof}[Proof of Theorem \ref{thm:grad_1}] Let $\theta\in (0,1)$ be chosen later, $\varphi\in \mathrm{C}_c^\infty(\Omega)$, $0\leq \varphi\leq 1$, $\mathrm{supp}\;\varphi\subset \Omega$ and $\varphi = 1$ on $\Omega_\delta$ such that
\begin{equation}\label{e:ass_power}
    |\Delta \varphi(x)| \leq C\varphi^\theta \qquad\text{and}\qquad |D \varphi(x)|^2 \leq C\varphi^{1+\theta},\quad \forall x\in\Omega,
\end{equation}
where $C = C(\delta,\theta)$ is a constant depending on $\delta,\theta$.

Define $w(x) := |Du^\varepsilon(x)|^2$ for $x \in \Omega$. The equation for $w$ is given by
\begin{equation*}
    -\varepsilon \Delta w + 2 p|D u^\varepsilon|^{p-2}D u^\varepsilon \cdot D w + 2  w - 2 D f\cdot D u^\varepsilon + 2 \varepsilon |D^2u^\varepsilon|^2 = 0 \qquad\text{in}\;\Omega.
\end{equation*}
%\textcolor{orange}{(I think the coefficients here are not correct. I put the corrected ones in orange. Originally, all the coefficients were 1 except for the term 2w.)}

Then an equation for $(\varphi w)$ can be derived as follows.
\begin{align*}
    &-\varepsilon \Delta (\varphi w) + 2 p|D u^\varepsilon|^{p-2}D u^\varepsilon \cdot D (\varphi w) + 2  (\varphi w) + 2 \varepsilon \varphi|D^2u^\varepsilon|^2 + 2\varepsilon \frac{D \varphi}{\varphi}\cdot D (\varphi w) \\
     = &\varphi(D f\cdot D u^\varepsilon) + 2 p|D u^\varepsilon|^{p-2}(D u^\varepsilon\cdot D \varphi)w -\varepsilon w \Delta \varphi + 2\varepsilon \frac{|D \varphi|^2}{\varphi}w
    \qquad\text{in}\;\mathrm{supp}\;\varphi.
\end{align*}
Assume that $\varphi w$ achieves its maximum over $\overline{\Omega}$ at $x_0\in \Omega$. And we can further assume that $x_0\in \mathrm{supp}\;\varphi$, since otherwise the maximum of $\varphi w$ over $\overline{\Omega}$ is zero. By the classical maximum principle,
\begin{equation*}
    -\varepsilon \Delta(\varphi w)(x_0)\geq 0 \qquad\text{and}\qquad |D(\varphi w)(x_0)| = 0.
\end{equation*}
Use this in the equation of $\varphi w$ above to obtain
\begin{equation*}
    \varepsilon \varphi|D^2u^\varepsilon|^2 \leq  \varphi (Df\cdot Du^\varepsilon)+ 2 p|Du^\varepsilon|^{p-1} |D\varphi|w + \varepsilon w |\Delta\varphi|  + 2\varepsilon  w\frac{|D\varphi|^2}{\varphi},
\end{equation*}
 where all terms are evaluated at $x_0$. From \eqref{e:ass_power}, we have
\begin{equation}\label{e:est_for_D^2u}
    \varepsilon \varphi|D^2u^\varepsilon|^2 \leq  \varphi |Df|w^{\frac{1}{2}}+ 2 Cp w^{\frac{p-1}{2}+1} \varphi^{\frac{1+\theta}{2}} + C\varepsilon w \varphi^{\theta} + 2C\varepsilon  w\varphi^\theta.
\end{equation}
By Cauchy-Schwartz inequality, $n|D^2u^\varepsilon|^2\geq (\Delta u^\varepsilon)^2$. Thus, if $n\varepsilon < 1$, then
\begin{equation}\label{e:est_for_D^2u_2}
    \varepsilon |D^2u^\varepsilon|^2 \geq \frac{(\varepsilon \Delta u^\varepsilon)^2}{n\varepsilon} \geq (\varepsilon \Delta u^\varepsilon)^2 = \left(  u^\varepsilon + |Du^\varepsilon|^p - f\right)^2 \geq |Du^\varepsilon|^{2p} - 2C|Du^\varepsilon|^p \geq \frac{|Du^\varepsilon|^{2p}}{2} - 2C,
\end{equation}
where $C$ depends on $\max_{\overline{\Omega}}f$ only. Using \eqref{e:est_for_D^2u_2} in \eqref{e:est_for_D^2u}, we obtain that
\begin{equation*}
    \varphi\left(\frac{1}{2}w^p - 2C\right) \leq \varphi |Df|w^{\frac{1}{2}}+ 2 Cp w^{\frac{p-1}{2}+1} \varphi^{\frac{1+\theta}{2}} + 3C\varepsilon w \varphi^{\theta}.
\end{equation*}
Multiply both sides by $\varphi^{p-1}$ to deduce that
\begin{align*}
    (\varphi w)^p \leq 4C\varphi^{p-1} + 2\Vert Df\Vert_{L^\infty}\varphi^p w^{\frac{1}{2}} + 4Cp \varphi^{\frac{2p+\theta - 1}{2}}w^{\frac{p+1}{2}} + 6C\varepsilon \varphi^{p+\theta - 1}w.
\end{align*}
Choose $2p+\theta -1 \geq p+1$, i.e., $p+\theta\geq 2$. This is always possible with the requirement $\theta \in (0,1)$, as $1<p <\infty$. Then we get
\begin{equation}\label{e:est_for_D^2u_3}
    (\varphi w)^p \leq C\left(1+ (\varphi w)^\frac{1}{2} + (\varphi w)^\frac{p+1}{2} +(\varphi w)\right).
\end{equation}
As a polynomial in $z = (\varphi w)(x_0)$, this implies that $(\varphi w)(x_0)\leq C$ where $C$ depends on coefficients of the right hand side of \eqref{e:est_for_D^2u_3}, which gives our desired gradient bound since $w(x)=(\varphi w)(x) \leq (\varphi w)(x_0)$ for $x \in \overline{\Omega}_\delta\subset \mathrm{supp}\;\varphi$.
\end{proof}

\subsection{Well-posedness of \eqref{eq:PDEeps}}
\begin{proof} [Proof of Theorem \ref{thm:wellposed1<p<2}] If $p\in (1,2)$, we use the ansatz $ u(x) = C_\varepsilon d(x)^{-\alpha}$ to find a solution to \eqref{eq:PDEeps}. Plug the ansatz into \eqref{eq:PDEeps} and compute
\begin{equation*}
\begin{split}
    |Du (x)|^p &= \frac{(\alpha C_\varepsilon)^p }{d(x)^{p(\alpha+1)}}|D d(x)|^p,\\
    \varepsilon\Delta u(x) &= \frac{\varepsilon C_\varepsilon\alpha(\alpha+1)}{d(x)^{\alpha+2}}|D d(x)|^2 - \frac{\varepsilon C_\varepsilon\alpha}{d(x)^{\alpha+1}}\Delta d(x).
\end{split}
\end{equation*}
Since $|D d(x)| = 1$ for $x$ near $\partial\Omega$, as $x\to \partial \Omega$, the explosive terms of the highest order are
\begin{equation*}
         C_\varepsilon^p \alpha^p d^{-(\alpha+1)p}  -\varepsilon C_\varepsilon \alpha(\alpha+1)d^{-(\alpha+2)}.
\end{equation*}
Set the above to be zero to obtain that
\begin{equation}\label{e:relation}
    \displaystyle\alpha = \frac{2-p}{p-1} \qquad\text{and}\qquad C_\varepsilon = \left(\frac{1}{\alpha}(\alpha+1)^\frac{1}{p-1}\right) \varepsilon^{\frac{1}{p-1}} = \frac{1}{\alpha}(\alpha+1)^{\alpha+1}\varepsilon^{\alpha+1}.
\end{equation}
For $0<\delta < \frac{1}{2}\delta_0$ and $\eta$ small, define
\begin{equation*}
\begin{split}
    \overline{w}_{\eta,\delta}(x) &:= \frac{(C_\alpha+\eta)\varepsilon^{\alpha+1}}{(d(x)-\delta)^\alpha} + M_\eta, \qquad x\in \Omega_\delta,\\
    \underline{w}_{\eta,\delta}(x) &:= \frac{(C_\alpha-\eta)\varepsilon^{\alpha+1}}{(d(x)+\delta)^\alpha} - M_\eta, \qquad x\in \Omega^\delta,
\end{split}
\end{equation*}
where $C_\alpha := \frac{1}{\alpha} (\alpha+1)^{\alpha+1} $, $M_\eta$ to be chosen. Next, we show that $\overline{w}_{\eta,\delta}$ is a supersolution of \eqref{eq:PDEeps} in $\Omega_\delta$, while $\underline{w}_{\eta,\delta}$ is a subsolution of \eqref{eq:PDEeps} in $\Omega^\delta$. Compute
\begin{align*}
    \mathcal{L}^\varepsilon\left[\overline{w}_{\eta,\delta}\right](x) = & \frac{ (C_\alpha + \eta)\varepsilon^{\alpha+1}}{(d(x)-\delta)^\alpha} + M_\eta + \frac{(C_\alpha+\eta)^p \alpha^p\varepsilon^{\alpha+2}}{(d(x)-\delta)^{\alpha+2}}|D d(x)|^p - f(x) \\
    & - \frac{(C_\alpha+\eta)\alpha(\alpha+1)\varepsilon^{\alpha+2}}{(d(x)-\delta)^{\alpha+2}}|D d(x)|^2 + \frac{(C_\alpha+\eta)\alpha \varepsilon^{\alpha+2}}{(d(x)-\delta)^{\alpha+1}}\Delta d(x)\\
    \geq &M_\eta - f(x) + \underbrace{\frac{\nu C_\alpha\alpha(\alpha+1)\varepsilon^{\alpha+2}}{(d(x)-\delta)^{\alpha+2}}\left[\nu^{p-1}|Dd(x)|^{p} - |Dd(x)|^2 + \frac{(d(x)-\delta)\Delta d(x)}{\alpha+1}\right]}_{I},
\end{align*}
%\textcolor{orange}{(In the parenthesis of I, I think the first term should be $\nu^{p-1}|Dd(x)|^p$ instead of $\nu^{p-1}|Dd(x)|^{p-1}$.)}
where we use $(C_\alpha\alpha)^p = C_\alpha\alpha(\alpha+1)$ and $\nu = \frac{C_\alpha+\eta}{C_\alpha} \in (1,2)$ for small $\eta$. Let
\begin{equation*}
    \delta_\eta : = \frac{\alpha+1}{K_2}\left[\nu^{p-1}-1\right]
\end{equation*}
and $\delta_\eta \to 0$ as $\eta \to 0$.
Recall the definitions in \eqref{boundond}. To get  $\mathcal{L}^\varepsilon\left[\overline{w}_{\eta,\delta}\right]\geq 0$, there are two cases to consider, depending on how large $d(x)-\delta$ is.
\begin{itemize}
    \item If $0< d(x)-\delta <\delta_\eta < \delta_0$ for $\eta$ small and fixed, then $|Dd(x)| = 1$, and thus $I\geq 0$. Hence, $\mathcal{L}^\varepsilon\left[\overline{w}_{\eta,\delta}\right]\geq 0$ if we choose $M_\eta \geq \max_{\overline{\Omega}} f$.
    %\textcolor{orange}{(Question: Here, $\eta$ is fixed right?)}
    \item If $d(x)-\delta\geq \delta_\eta$, then
    \begin{equation*}
        I \leq \left(\frac{1}{\delta_\eta}\right)^{\alpha+2}\nu C_\alpha \alpha(\alpha+1)\left[\nu^{p-1}K_1^{p}+K_1^2+K_2K_0\right]\varepsilon^{\alpha+2}.
    \end{equation*}
    %\textcolor{orange}{(The first term in the bracket should be $\nu^{p-1}K_1^p$ instead of $\nu^{p-1}K_1^{p-1}$.)}

    Thus, we can choose $M_\eta = \max_{\overline{\Omega}} f + C\varepsilon^{\alpha+2}$ for $C$ large enough (depending on $\eta$) so that  $\mathcal{L}^\varepsilon\left[\overline{w}_{\eta,\delta}\right]\geq 0$.
\end{itemize}
Therefore, $\overline{w}_{\eta,\delta}$ is a supersolution in $\Omega_\delta$.

Similarly, we have
\begin{align*}
    \mathcal{L}_\varepsilon\left[\underline{w}_{\eta,\delta}\right](x) = &\frac{ (C_\alpha - \eta)\varepsilon^{\alpha+1}}{(d(x)+\delta)^\alpha} - M_\eta + \frac{(C_\alpha-\eta)^p \alpha^p\varepsilon^{\alpha+2}}{(d(x)+\delta)^{\alpha+2}}|D d(x)|^p - f(x) \\
    & - \frac{(C_\alpha-\eta)\alpha(\alpha+1)\varepsilon^{\alpha+2}}{(d(x)+\delta)^{\alpha+2}}|D d(x)|^2 + \frac{(C_\alpha-\eta)\alpha \varepsilon^{\alpha+2}}{(d(x)+\delta)^{\alpha+1}}\Delta d(x)\\
    =& - M_\eta - f(x) \\
    &+\underbrace{\frac{\nu C_\alpha\alpha(\alpha+1)\varepsilon^{\alpha+2}}{(d(x)+\delta)^{\alpha+2}}\left[\nu^{p-1}|Dd(x)|^{p} - |Dd(x)|^2 + \frac{(d(x)+\delta)\Delta d(x)}{\alpha+1}+\frac{(d(x)+\delta)^2}{\alpha(\alpha+1)\varepsilon}\right]}_{J},
\end{align*}
%\textcolor{orange}{(In the parenthesis of J, the first term should be $\nu^{p-1}|Dd(x)|^p$ instead of $\nu^{p-1}|Dd(x)|^{p-1}$.)}

where $\nu = \frac{C_\alpha-\eta}{C_\alpha}\in (0,1)$ for small $\eta$. Let
\begin{equation*}
    \delta_\eta := \left(1-\nu^{p-1}\right)\left(\frac{\alpha(\alpha+1)\varepsilon}{1+K_2\alpha\varepsilon}\right)
\end{equation*}
and $\delta_\eta \to 0$ as $ \eta \to 0$.
To obtain $\mathcal{L}^\varepsilon\left[\underline{w}_{\eta,\delta}\right]\leq 0$, there are two cases to consider depending on how large $d(x)+\delta$ is.
\begin{itemize}
    \item If $0<d(x)+\delta<\delta_\eta < \delta_0$ for $\eta$ small and fixed, then $|Dd(x)| = 1$, and thus $J\leq 0$. Hence, $\mathcal{L}^\varepsilon\left[\underline{w}_{\eta,\delta}\right]\leq 0$ if we choose $M_\eta \geq -\max_{\Omega}f$.
    \item If $d(x)+\delta\geq \delta_\eta$, then
    \begin{equation*}
        |J|\leq \left(\frac{1}{\delta_\eta}\right)^{\alpha+2} \nu C_\alpha\alpha(\alpha+1)\left[\nu^{p-1}K_1^{p}+K_1^2 + \frac{(K_0+1)K_2}{\alpha+1} + \frac{(K_0+1)^2}{\alpha(\alpha+1)\varepsilon}\right]\varepsilon^{\alpha+2}
    \end{equation*}
    %\textcolor{orange}{($|J|\leq \left(\frac{1}{\delta_\eta}\right)^{\alpha+2} \nu C_\alpha\alpha(\alpha+1)\left[K_1^{p}+K_1^2 + \frac{(K_0+\delta)K_2}{\alpha+1} + \frac{(K_0+\delta)^2}{\alpha(\alpha+1)\varepsilon}\right]\varepsilon^{\alpha+2}$)}

     Thus, we can choose $M_\eta = -\max_{\overline{\Omega}} f - C\varepsilon^{\alpha+2}$ for $C$ large enough (depending on $\eta$) so that $\mathcal{L}^\varepsilon\left[\underline{w}_{\eta,\delta}\right]\leq 0$.
\end{itemize}
Therefore, $\underline{w}_{\eta,\delta}$ is a subsolution in $\Omega^\delta$.

For $p=2$, we use the ansatz $u(x) = -C_\varepsilon \log(d(x))$ instead. Similar to the previous case, one can find $u(x) = -\varepsilon\log(d(x))$. For $0<\delta<\frac{1}{2}\delta_0$,  define
\begin{equation*}
    \begin{split}
        &\overline{w}_{\eta,\delta}(x) = -(1+\eta)\varepsilon\log(d(x)-\delta) + M_\eta, \qquad x\in \Omega_\delta,\\
        &\underline{w}_{\eta,\delta}(x) = -(1-\eta)\varepsilon\log(d(x)+\delta) - M_\eta, \qquad x\in \Omega^\delta,
    \end{split}
\end{equation*}
where $M_\eta$ is to be chosen so that $\overline{w}_{\eta,\delta}(x)$ is a supersolution in $\Omega_\delta$ and $\underline{w}_{\eta,\delta}$ is a subsolution in $\Omega^\delta$. The computations are omitted here, as they are similar to the previous case.
\smallskip

\noindent We divide the rest of the proof into 3 steps. We first construct a minimal solution, then a maximal solution to \eqref{eq:PDEeps}, and finally show that they are equal to conclude the existence and the uniqueness of the solution to \eqref{eq:PDEeps}.
\smallskip

\paragraph{\textbf{Step 1.}} There exists a minimal solution $\underline{u}\in \mathrm{C}^2(\Omega)$ of \eqref{eq:PDEeps} such that $v\geq \underline{u}$ for any other solution $v\in \mathrm{C}^2(\Omega)$ solving \eqref{eq:PDEeps}.

\begin{proof} Let $w_{\eta,\delta}\in \mathrm{C}^2(\Omega)$ solve
    \begin{equation}\label{e:w_def}
    \begin{cases}
        \mathcal{L}^\varepsilon\left[w_{\eta,\delta}\right] = 0 &\qquad\text{in}\;\Omega,\\
        \qquad w_{\eta,\delta} = \underline{w}_{\eta,\delta} &\qquad\text{on}\;\partial\Omega.
    \end{cases}
    \end{equation}
    \begin{itemize}
        \item Fix $\eta>0$. As $\delta\to 0^+$, the value of $\underline{w}_{\eta,\delta}$ blows up on the boundary. Therefore, by the standard comparison principle for the second-order elliptic equation with the Dirichlet boundary, $\delta_1 \leq  \delta_2$ implies $w_{\eta,\delta_1}\geq  w_{\eta,\delta_2}$ on $\overline{\Omega}$.
        \item For $\delta'>0$, since $\underline{w}_{\eta,\delta'}$ is a subsolution in $\overline{\Omega}$ with finite boundary,
            \begin{equation}\label{e:cp_delta1}
                0<\delta \leq \delta'\qquad\Longrightarrow\qquad \underline{w}_{\eta,\delta'} \leq w_{\eta_,\delta'}\leq w_{\eta,\delta} \qquad\text{on}\;\overline{\Omega}.
            \end{equation}
        \item Similarly, since $\overline{w}_{\eta,\delta'}$ is a supersolution on $\Omega_{\delta'}$ with infinity value on the boundary $\partial\Omega_{\delta'}$, by the comparison principle,
            \begin{equation}\label{e:cp_delta2}
                w_{\eta,\delta} \leq \overline{w}_{\eta, \delta'} \qquad\text{in}\;\Omega_{\delta'} \qquad\Longrightarrow\qquad w_{\eta,\delta} \leq \overline{w}_{\eta,0} \qquad\text{in}\;\Omega.
            \end{equation}
    \end{itemize}
    \noindent From \eqref{e:cp_delta1} and \eqref{e:cp_delta2}, we have
    \begin{equation}\label{e:cp_delta3}
        0<\delta \leq \delta'\qquad\Longrightarrow\qquad \underline{w}_{\eta,\delta'} \leq w_{\eta_,\delta'}\leq w_{\eta,\delta} \leq \overline{w}_{\eta,0} \qquad\text{in}\;\Omega.
    \end{equation}
    Thus, $\{w_{\eta,\delta}\}_{\delta>0}$ is locally bounded in $L^{\infty}_{\mathrm{loc}}(\Omega)$ ($\{w_{\eta,\delta}\}_{\delta>0}$ is uniformly bounded from below). Using the local gradient estimate for $w_{\eta,\delta}$ solving \eqref{e:w_def}, we deduce that $\{w_{\eta,\delta}\}_{\delta>0}$ is locally bounded in $W^{1,\infty}_{\mathrm{loc}}(\Omega)$. Since $w_{\eta,\delta}$ solves \eqref{e:w_def}, we further have that $\{w_{\eta,\delta}\}_{\delta>0}$ is locally bounded in $W^{2,r}_{\mathrm{loc}}(\Omega)$ for all $r<\infty$ by Calderon-Zygmund estimates.

    \noindent Local boundedness of $\{w_{\eta,\delta}\}_{\delta>0}$ in $W^{2,r}_{\mathrm{loc}}(\Omega)$ implies weak$^*$ compactness, that is, there exists a function $u\in W^{2,r}_{\mathrm{loc}}(\Omega)$ such that (via subsequence and monotonicity)
    \begin{equation*}
        w_{\eta,\delta} \rup u \qquad\text{weakly in}\;W^{2,r}_{\mathrm{loc}}(\Omega),\qquad \text{and}\qquad
        w_{\eta,\delta} \to u \qquad\text{strongly in}\;W^{1,r}_{\mathrm{loc}}(\Omega).
    \end{equation*}
    In particular, $w_{\eta,\delta}\to u$ in $\mathrm{C}^1_{\mathrm{loc}}(\Omega)$ thanks to Sobolev compact embedding. Let us rewrite the equation $\mathcal{L}^\varepsilon\left[w_{\eta,\delta}\right] = 0$ as $\varepsilon\Delta w_{\eta,\delta}(x) = F[w_{\eta,\delta}](x)$ for $x \in U\subset\subset \Omega$, where
    \begin{equation*}
        F[w_{\eta,\delta}](x) =    w_{\eta,\delta}(x) + H(x,Dw_{\eta,\delta}(x)).
    \end{equation*}
    Since $w_{\eta,\delta}\to u$ in $\mathrm{C}^1(U)$ as $\delta \to 0$, we have $F[w_{\eta,\delta}](x) \to F(x)$ uniformly in $U$ as $\delta \to 0$,  where
    \begin{equation*}
        F(x) =   u(x) + H(x,Du(x)).
    \end{equation*}
    In the limit, we obtain that $u\in L^2(U)$ is a weak solution of $\varepsilon\Delta u = F$ in $U$ where $F$ is continuous. Thus, $u\in \mathrm{C}^2(\Omega)$ and by stability, $u$ solves $\mathcal{L}^\varepsilon[u] = 0$ in $\Omega$. From \eqref{e:cp_delta3}, we also have
    \begin{equation*}
        \underline{w}_{\eta,0} \leq u \leq \overline{w}_{\eta,0} \qquad\text{in}\;\Omega.
    \end{equation*}
    Moreover, $u(x)\to \infty$ as $\mathrm{dist}(x,\partial\Omega)\to 0$ with the precise rate like \eqref{rate_p<2} or \eqref{rate_p=2}. Note that by construction, $u$ may depend on $\eta$. But next, we will show that $u$ is independent of $\eta$, by proving $u$ is the unique minimal solution of $\mathcal{L}^\varepsilon[u] = 0$ in $\Omega$ with $u = +\infty$ on $\partial\Omega$.

    \noindent Let $v\in \rmC^{2}(\Omega)$ be a solution to \eqref{eq:PDEeps}. Fix $\delta>0$. Since $v(x)\to \infty$ as $x\to \partial\Omega$ while $w_{\eta,\delta}$ remains bounded on $\partial \Omega$, the comparison principle yields
    \begin{equation*}
        v\geq w_{\eta,\delta} \qquad\text{in} \; \Omega.
    \end{equation*}
    Let $\delta\to 0$ and we deduce that $v\geq u$ in $\Omega$. This concludes that $u$ is the minimal solution in $\mathrm{C}^2(\Omega)(\forall\,r<\infty)$ and thus $u$ is independent of $\eta$.
\end{proof}

\paragraph{\textbf{Step 2.}} There exists a maximal solution $\overline{u}\in \mathrm{C}^2(\Omega)$ of \eqref{eq:PDEeps} such that $v\leq \overline{u}$ for any other solution $v\in \mathrm{C}^2(\Omega)$ solving \eqref{eq:PDEeps}.

\begin{proof} For each $\delta>0$, let $u_\delta\in \mathrm{C}^2(\Omega_\delta)$ be the minimal solution to $\mathcal{L}^\varepsilon[u_\delta] = 0$ in $\Omega_\delta$ with $u_\delta = +\infty$ on $\partial\Omega_\delta$. By the comparison principle, for every $\eta>0$, there holds
\begin{equation*}
    \underline{w}_{\eta,\delta} \leq u_\delta \leq \overline{w}_{\eta,\delta} \qquad\text{in}\;\Omega_\delta,
\end{equation*}
and
\begin{equation*}
    0<\delta<\delta' \qquad \Longrightarrow\qquad u_\delta \leq u_\delta' \qquad\text{in}\;\Omega_{\delta'}\,.
\end{equation*}
The monotoniciy, together with the local boundedness of $\{u_\delta\}_{\delta>0}$ in $W^{2,r}_{\mathrm{loc}}(\Omega)$, implies that there exists $u\in W^{2,r}_{\mathrm{loc}}(\Omega)$ for all $r<\infty$ such that $u_\delta\to u$ strongly in $\rmC^1_{\mathrm{loc}}(\Omega)$. Using the equation $\mathcal{L}^\varepsilon [u_\delta] = 0$ in $\Omega_\delta$ and the regularity of Laplace's equation, we can further deduce that $u\in \mathrm{C}^2(\Omega)$ solves \eqref{eq:PDEeps} and
\begin{equation*}
    \underline{w}_{\eta,0} \leq u\leq \overline{w}_{\eta,0} \qquad\text{in}\;\Omega
\end{equation*}
for all $\eta>0$. As $u_\delta$ is independent of $\eta$ by the previous argument in Step 1, it is clear that $u$ is also independent of $\eta$. Now we show that $u$ is the maximal solution of \eqref{eq:PDEeps}. Let $v\in\rmC^2(\Omega)$ solve \eqref{eq:PDEeps}. Clearly $v\leq u_\delta$ on $\Omega_\delta$. Therefore, as $\delta \to 0$, we have $v\leq u$.
\end{proof}
\noindent In conclusion, we have found a minimal solution $\underline{u}$ and a maximal solution $\overline{u}$ in $\rmC^2(\Omega)$ such that
\begin{equation}\label{e:chain}
    \underline{w}_{\eta,0} \leq \underline{u}\leq \overline{u}\leq \overline{w}_{\eta,0} \qquad\text{in}\;\Omega
\end{equation}
for any $\eta>0$. This extra parameter $\eta$ now enables us to show that $\overline{u} = \underline{u}$ in $\Omega$. The key ingredient here is the convexity in the gradient slot of the operator.
\smallskip
\paragraph{\textbf{Step 3.}} We have $\overline{u}\equiv \underline{u}$ in $\Omega$. Therefore, the solution to \eqref{eq:PDEeps} in $\mathrm{C}^2(\Omega)$ is unique.

\begin{proof} Let $\theta\in (0,1)$. Define $w_\theta = \theta \overline{u} + (1-\theta) \inf_{\Omega} f$. It can be verified that $w_\theta$ is a subsolution to \eqref{eq:PDEeps}. Then one may argue that by the comparison principle,
\begin{equation*}
    w_\theta = \theta \overline{u} + (1-\theta)\inf_{\Omega} f\leq \underline{u} \qquad\text{in}\;\Omega,
\end{equation*}
and conclude that $\overline{u} \leq \underline{u}$ by letting $\theta\to 1$. But we have to be careful here. As they are both explosive solutions, to use the comparison principle, we need to show that $w_\theta \leq \underline{u}$ in a neighborhood of $\partial\Omega$. From \eqref{e:chain}, we see that
\begin{align*}
    &1\leq \frac{\overline{u}(x)}{\underline{u}(x)} \leq \frac{\overline{w}_{\eta,0}(x)}{\underline{w}_{\eta,0}(x)} = \frac{(C_\alpha+\eta)+ M_\eta d(x)^\alpha}{(C_\alpha-\eta)- M_\eta d(x)^\alpha},& 1<p<2,\\
    &1 \leq \frac{\overline{u}(x)}{\underline{u}(x)} \leq \frac{\overline{w}_{\eta,0}(x)}{\underline{w}_{\eta,0}(x)} = \frac{-(1+\eta)\log(d(x)) + M_\eta}{-(1-\eta)\log(d(x))-M_\eta}, & p=2,
\end{align*}
for $x\in \Omega$. Hence,
\begin{align*}
     &1\leq  \lim_{d(x)\to 0} \left(\frac{\overline{u}(x)}{\underline{u}(x)}\right) \leq \frac{C_\alpha+\eta}{C_\alpha-\eta}, & 1< p < 2,\\
     &1\leq  \lim_{d(x)\to 0} \left(\frac{\overline{u}(x)}{\underline{u}(x)}\right) \leq \frac{-(1+\eta)}{-(1-\eta)}, & p = 2.
\end{align*}
Since $\eta>0$ is chosen arbitrary, we obtain
\begin{equation*}
     \lim_{d(x)\to 0} \left(\frac{\overline{u}(x)}{\underline{u}(x)}\right) = 1.
\end{equation*}
 This means for any $\varsigma\in(0,1)$, there exists ${\delta_1}(\varsigma)>0$ small such that
\begin{equation*}
\frac{\overline{u}(x)}{\underline{u}(x)}\leq (1+\varsigma)\Longrightarrow \left(\frac{1}{1+\varsigma}\right)\overline{u}(x) \leq \underline{u}(x) \qquad\text{in}\; \Omega\backslash \Omega_{\delta_1}.
\end{equation*}
For a fixed $\theta\in (0,1)$, one can always choose $\varsigma$ small enough so that $\displaystyle (1+\varsigma)^{-1} \geq   \frac{1+\theta}{2}$. Since $\overline{u}(x) \to +\infty$ as $d(x) \to 0$, there exists $\delta_2 > 0$ such that $\overline{u}(x) \geq 2 \inf_\Omega f$ for all $x \in \Omega \setminus \Omega_{\delta_2}$.
Now we have
\begin{equation*}
    \underline{u}(x)\geq \left(\frac{1}{1+\varsigma}\right)\overline{u}(x) \geq \theta \overline{u}(x)+  \left( \frac{1-\theta}{2}\right)\overline{u}(x)\geq \theta \overline{u}(x) + (1-\theta)\left(\inf_\Omega f\right)
\end{equation*}
 for all $x \in \Omega \setminus \Omega_\delta$ where $\delta: = \min \{\delta_1, \delta_2\}$. This implies for any fixed $\theta\in(0,1)$, $w_\theta \leq \underline{u}$ in a neighborhood of $\partial\Omega$. Hence, by the comparison principle,
\begin{equation*}
    w_\theta = \theta \overline{u} + (1-\theta)\inf_{\Omega} f\leq \underline{u} \qquad\text{in}\;\Omega,
\end{equation*}
for any $\theta \in (0,1)$. Then let $\theta \to 1$ to get the conclusion.
\end{proof}

\noindent This finishes the proof of the well-posedness of \eqref{eq:PDEeps} for $1<p\leq2$.
\end{proof}

\begin{proof}[Proof of Lemma \ref{lem:max}] The proof is a variation of Perron's method (see \cite{Capuzzo-Dolcetta1990}) and we proceed by contradiction. Let $\varphi\in \rmC(\overline{\Omega})$ and $x_0\in \overline{\Omega}$ such that $u(x_0) = \varphi(x_0)$ and $u-\varphi$ has a global strict minimum over $\overline{\Omega}$ at $x_0$ with
\begin{equation}\label{eq:max_a1}
      \varphi(x_0) + H(x_0,D\varphi(x_0)) < 0.
\end{equation}
Let $\varphi^\varepsilon(x) = \varphi(x) - |x-x_0|^2 + \varepsilon$ for $x\in \overline{\Omega}$. Let $\delta > 0$. We see that for $x\in \partial B(x_0,\delta)\cap \overline{\Omega}$,
\begin{equation*}
    \varphi^\varepsilon(x) = \varphi(x) - \delta^2 +\varepsilon \leq \varphi(x) - \varepsilon
\end{equation*}
if $2\varepsilon \leq \delta^2$. We observe that
\begin{equation*}
    \begin{split}
    \varphi^\varepsilon(x) - \varphi(x_0)  &= \varphi(x)-\varphi(x_0) + \varepsilon - |x-x_0|^2 \\
    D\varphi^\varepsilon(x) - D\varphi(x_0) &= D\varphi(x) - D\varphi(x_0) - 2(x-x_0)
    \end{split}
\end{equation*}
for $x\in B(x_0,\delta)\cap \overline{\Omega}$. By the continuity of $H(x,p)$ near $(x_0,D\varphi(x_0))$ and the fact that $\varphi\in \rmC^1(\overline{\Omega})$, we can deduce from \eqref{eq:max_a1} that if $\delta$ is small enough and $0<2\varepsilon < \delta^2$, then
\begin{equation}\label{eq:max_a2}
      \varphi^\varepsilon(x)+H(x,D\varphi^\varepsilon(x)) < 0 \qquad\text{for}\;x\in B(x_0,\delta)\cap \overline{\Omega}.
\end{equation}
We have found $\varphi^\varepsilon\in \mathrm{C}^1(\overline{\Omega})$ such that $\varphi^\varepsilon(x_0)>u(x_0)$, $\varphi^\varepsilon<u$ on $\partial B(x_0,\delta)\cap \overline{\Omega}$ and \eqref{eq:max_a2}. Let
\begin{equation*}
    \tilde{u}(x) = \begin{cases}
    \max \big\lbrace u(x),\varphi^\varepsilon(x) \big\rbrace &x\in B(x_0,\delta)\cap \overline{\Omega},\\
    u(x)&x\notin B(x_0,\delta)\cap \overline{\Omega}.\\
    \end{cases}
\end{equation*}
We see that $\tilde{u}\in \rmC(\overline{\Omega})$ is a subsolution of \eqref{eq:PDE0} in $\Omega$ with $\tilde{u}(x_0) > u(x_0)$, which is a contradiction. Thus, $u$ is a supersolution of \eqref{eq:PDE0} on $\overline{\Omega}$.
\end{proof}

\subsection{Semiconcavity}
We present a proof for the semiconcavity of solution to first-order Hamilton--Jacobi equation using the doubling variable method (see also \cite{Calder2021}).
\begin{thm}\label{convex} Let $H(x,p) = G(p)-f(x)$ where $G\geq 0$ with $G(0) = 0$ is a convex function from $\R^n\to\R^n$ and $f\in \mathrm{C}^2_c(\R^n)$. Let $u\in \mathrm{C}_c(\R^n)$ be a viscosity solution to $u+H(x,Du) = 0$ in $\R^n$. Then $u$ is semiconcave, i.e., $u$ is a viscosity solution of  $-D^2u \succ -c\;\mathbb{I}_n$ in $\R^n$ where
\begin{equation*}
    c = \max \big\lbrace  D_{\xi\xi}f(x): |\xi|=1, x\in \R^n \big\rbrace\geq 0.
\end{equation*}
\end{thm}
\begin{proof} Consider the auxiliary functional
\begin{equation*}
    \Phi(x,y,z) = u(x)-2u(y)+u(z) - \frac{\alpha}{2}|x-2y+z|^2 - \frac{c}{2}|y-x|^2-\frac{c}{2}|y-z|^2
\end{equation*}
for $(x,y,z)\in \R^n\times\R^n\times\R^n$. By the a priori estimate, $u$ is bounded and Lipschitz. Thus, we can assume $\Phi$ achieves its maximum over $\R^n\times\R^n\times \R^n$ at $(x_\alpha,y_\alpha,z_\alpha)$. The viscosity solution tests give us
\begin{align*}
    &u(x_\alpha) + G\big(p_\alpha+c(x_\alpha-y_\alpha)\big) \leq f(x_\alpha)\\
    &u(z_\alpha) + G\big(p_\alpha+c(z_\alpha-y_\alpha)\big) \leq f(z_\alpha)\\
    &u(y_\alpha) + G\left(p_\alpha+\frac{c}{2}(x_\alpha-y_\alpha) + \frac{c}{2}(z_\alpha-y_\alpha)\right) \geq f(y_\alpha),
\end{align*}
where $p_\alpha = \alpha(x_\alpha-2y_\alpha+z_\alpha)$. By the convexity of $G$, we have
\begin{equation*}
    2G\left(p_\alpha+\frac{c}{2}(x_\alpha-y_\alpha) + \frac{c}{2}(z_\alpha-y_\alpha)\right) \leq G\big(p_\alpha+c(x_\alpha-y_\alpha)\big)  + G\big(p_\alpha+c(z_\alpha-y_\alpha)\big)
\end{equation*}
Therefore,
\begin{equation*}
    u(x_\alpha) - 2u(y_\alpha) + u(z_\alpha) \leq f(x_\alpha) - 2f(y_\alpha) + f(z_\alpha).
\end{equation*}
\begin{itemize}
    \item $\Phi(x_\alpha,y_\alpha,z_\alpha)\geq \Phi(0,0,0)$ gives us
    \begin{equation*}
        \frac{\alpha}{2}|x_\alpha-2y_\alpha+z_\alpha|^2 + \frac{c}{2}|y_\alpha-x_\alpha|^2+\frac{c}{2}|y_\alpha-z_\alpha|^2 \leq C.
    \end{equation*}
    Thus, $(x_\alpha-y_\alpha)\to h_0$ and $(y_\alpha-z_\alpha)\to h_0$ as $\alpha\to \infty$ for some $h_0 \in \mathbb{R}^n$.

    \item $\Phi(x_\alpha,y_\alpha,z_\alpha )\geq \Phi(y_\alpha+h_0,y_\alpha,y_\alpha-h_0)$ gives us
    \begin{align*}
        u(x_\alpha)-2u(y_\alpha) + u(z_\alpha) - \frac{\alpha}{2}|x_\alpha - 2y_\alpha+z_\alpha|^2 - \frac{c}{2}|x_\alpha-y_\alpha|^2 - \frac{c}{2}|y_\alpha-z_\alpha|^2 \\
        \geq u(y_\alpha+h_0) - 2u(y_\alpha) + u(y_\alpha-h_0) - c|h_0|^2.
    \end{align*}
    Therefore, by the fact that $u$ is Lipschitz, we have
    \begin{equation*}
    \begin{split}
        \frac{\alpha}{2}|x_\alpha - 2y_\alpha+z_\alpha|^2
        \leq & c\left(\frac{2|h_0|^2 - |x_\alpha-y_\alpha|^2 - |y_\alpha-z_\alpha|^2}{2}\right) \\
        &+ C\Big(|(x_\alpha-y_\alpha) - h_0| + |(z_\alpha - y_\alpha) + h_0|\Big) \to 0
    \end{split}
    \end{equation*}
    as $\alpha\to \infty$.
\end{itemize}
For any $x\in \R^n$, we have $\Phi(x_\alpha,y_\alpha,z_\alpha) \geq \Phi(x+h,x,x-h)$, i.e.,
\begin{align*}
    u(x+h)-2u(x)+u(x-h)-c|h|^2 \leq &f(x_\alpha)-2f(y_\alpha) + f(z_\alpha) \\
    &- \frac{\alpha}{2}|x_\alpha - 2y_\alpha+z_\alpha|^2-\frac{c}{2}|y_\alpha-x_\alpha|^2-\frac{c}{2}|y_\alpha-z_\alpha|^2.
\end{align*}
If $\{y_\alpha\}$ is unbounded, then since $f\in \mathrm{C}_c^2(\R^n)$, we have $f(y_\alpha)\to 0$ as $\alpha\to \infty$. As a consequence, $x_\alpha,z_\alpha \to \infty$ as well and thus $f(x_\alpha)-2f(y_\alpha) + f(z_\alpha)\to 0$ as $\alpha\to \infty$. Therefore,
\begin{equation*}
    u(x+h)-2u(x)+u(x-h)-c|h|^2 \leq 0.
\end{equation*}
If $\{y_\alpha\}$ is bounded, then $y_\alpha\to y_0$ for some $y_0 \in \mathbb{R}^n$ as $\alpha\to \infty$. Thus,
\begin{equation*}
     u(x+h)-2u(x)+u(x-h)-c|h|^2 \leq f(y_0+h_0) - 2f(y_0) + f(y_0-h_0) -c|h_0|^2.
\end{equation*}
Let $\xi = h_0$ and we have
\begin{equation*}
    \begin{cases}
      f(y_0+h_0) - f(y_0) = \displaystyle\int_0^1 D_x f(y_0 + t\xi)\cdot \xi dt, \vspace{0.2cm}\\
      f(y_0) - f(y_0-h_0)  = \displaystyle\int_0^1 D_x f(y_0 - \xi + t\xi)\cdot \xi dt.
    \end{cases}
\end{equation*}
Therefore,
\begin{equation*}
\begin{split}
     f(y_0+h_0) - 2f(y_0) + f(y_0-h_0) &= \int_0^1 \Big(D_x f(y_0 + t\xi) - D_x f(y_0 - \xi + t\xi) \Big)\cdot \xi dt\\
     &=  \int_0^1 \int_0^1 \xi^\mathsf{T} D^2 f(y_0-\xi +t\xi+s \xi) \xi\;dsdt.
\end{split}
\end{equation*}
which implies
\begin{equation*}
    |f(y_0+h_0) - 2f(y_0) + f(y_0-h_0)| \leq \left(\max_{|\xi|=1}D_{\xi\xi} f\right)|\xi|^2.
\end{equation*}
Hence,
\begin{equation*}
    u(x+h)-2u(x)+u(x-h)-c|h|^2 \leq 0
\end{equation*}
and thus $u$ is semiconcave. It is easy to see that if $\varphi$ is smooth and $u-\varphi$ has a local min at $x$, then $D^2\varphi(x)\prec c\;\mathbb{I}$, i.e., $-D^2\varphi(x)\geq -c\;\mathbb{I}$.
\end{proof}

\begin{proof}[Proof of Theorem \ref{thm:newsemi}]\quad
\begin{itemize}
    \item[(i)] It is clear that
    \begin{equation*}
    \tilde{u}(x) = \begin{cases}
      u(x) &\qquad\text{if}\;x\in \overline{\Omega},\\
      0 &\qquad\text{if}\;x\notin \overline{\Omega},
    \end{cases}
    \end{equation*}
    solves the equation $\tilde{u}(x)+|D\tilde{u}(x)|^p - \tilde{f}(x) = 0$ in $\R^n$. By the exact same argument in the proof of Theorem \ref{convex}, the conclusion follows.
    %By Theorem \ref{convex}, we have $-D^2\tilde{u}\succ -c\;\mathbb{I}_n$ where $c$ is still the semiconcavity constant of $f$ in the whole space $\R^n$, thanks to the fact that $f\in \mathrm{C}_c^2(\R^n)$.
    \item[(ii)]
     %(see \cite[Chapter 5]{Lions1982_book})
Fix $x \in \Omega$ and let $\eta$ be a minimizing curve for $u(x)$. Then
$$u(x)=\int^\infty_0 e^{-s} \left(C_q|\dot{\eta}(s)|^q +f(\eta(s)) \right) ds .$$
Since $\eta(0)=x \in \Omega$, then there exists $T > 0$ such that $\eta(s) \in \Omega, \forall 0 \leq s \leq T$. In fact, we can choose $\displaystyle T \geq \frac{d(x)}{C_0}$ for some constant $C_0$ independent of $x$, since $\left\|\dot{\eta}\right\|_\infty \leq C$ where $C$ is independent of $x$.
Note that
\begin{equation}\label{x}
    u(x)=\int_0^{T} e^{-s}\left( C_q|\dot{\eta}(s)|^q +f(\eta(s)) \right)ds + e^{-T} u(\eta(T)).
\end{equation}
Define $\Tilde{\eta}:[0, +\infty)\to \mathbb{R}^n$ by
\begin{equation*}
    \Tilde{\eta}(s):=\left\{
    \begin{aligned}
    &\eta(s)+\left( 1-\frac{s}{T} \right) h, \quad \text{if }  0\leq s\leq T,\\
    &\eta(s),\qquad \qquad \qquad \quad  \text{if } s \geq T.
    \end{aligned}
    \right.
\end{equation*}
Choose $h$ small enough so that $\Tilde{\eta}(s) \in \Omega, \forall s\geq 0$.
(This can be done because there exists $r>0$ such that $B(\eta(s),r) \subset \Omega$, for all $0 \leq s \leq T$.)
By the optimal control formula of $u(x+h)$ and $u(x-h)$, we have
\begin{equation} \label{h+}
    u(x+h)\leq \int_0^{T} e^{-s}\left( C_q\left|\dot{\eta}(s)-\frac{h}{T}\right|^q +f \left(\eta(s)+\left(1-\frac{s}{T}\right)h\right) \right)ds + e^{-T} u(\eta(T)),
\end{equation}
and
\begin{equation}\label{h-}
    u(x-h)\leq \int_0^{T} e^{-s}\left( C_q \left|\dot{\eta}(s)+\frac{h}{T} \right|^q +f \left(\eta(s)-\left(1-\frac{s}{T} \right)h\right) \right)ds + e^{-T} u(\eta(T)).
\end{equation}

Hence, from \eqref{x}, \eqref{h+}, and \eqref{h-}, for $h$ small enough,
\begin{equation}\label{diff}
\begin{aligned}
    & u(x+h)+u(x-h)-2u(x) \\
    \leq & \int_0^T e^{-s} C_q \left( \left|\dot{\eta}(s)-\frac{h}{T}\right|^q + \left|\dot{\eta}(s)+\frac{h}{T} \right|^q -2\left|\dot{\eta}(s) \right|^q\right) ds\\
    &+\int_0^Te^{-s}\left(f \left(\eta(s)+\left(1-\frac{s}{T}\right)h\right) + f \left(\eta(s)-\left(1-\frac{s}{T} \right)h\right) -2f\left(\eta(s)\right) \right)ds\\
    \leq & \int_0^T e^{-s} C_q \left( \left|\dot{\eta}(s)-\frac{h}{T}\right|^q + \left|\dot{\eta}(s)+\frac{h}{T} \right|^q -2\left|\dot{\eta}(s)\right|^q\right) ds\\
    &+C |h|^2 \int_0^Te^{-s} \left( 1-\frac{s}{T}\right)^2 ds,
\end{aligned}
\end{equation}
where the second inequality follows from the semiconcavity of $f$.
By Taylor's theorem, for any $y \in \mathbb{R}^n$,
\begin{equation}
\begin{aligned}
     \left|y+\frac{1}{T}h\right|^q = &|y|^q+q|y|^{q-2}y \cdot \frac{1}{T} h +\int^1_0 q(q-2) \left|y +\frac{t}{T} h \right|^{q-4} \left( \left(y+\frac{t}{T}h\right) \cdot \frac{h}{T}\right)^2(1-t)dt\\
     &+\int^1_0q\left|y+\frac{t}{T}h\right|^{q-2}\left|\frac{h}{T}\right|^2(1-t)dt\\
     \leq &|y|^q+q|y|^{q-2}y \cdot \frac{1}{T} h + C\int^1_0 \left|y +\frac{t}{T} h \right|^{q-2}\left|\frac{h}{T}\right|^2dt\\
     \leq &|y|^q+q|x|^{q-2}y \cdot \frac{1}{T} h + C\left(|y|^{q-2}\left|\frac{h}{T}\right|^2+\left|\frac{h}{T}\right|^q \right)
\end{aligned}
\end{equation}
and similarly
\begin{equation}
    \begin{aligned}
    \left|y-\frac{1}{T}h\right|^q = &|y|^q-q|y|^{q-2}y \cdot \frac{1}{T} h + \int^1_0 q(q-2) \left|y-\frac{t}{T} h \right|^{q-4} \left( \left(y-\frac{t}{T}h\right) \cdot \frac{h}{T}\right)^2(1-t)dt\\
     &+\int^1_0q\left|y-\frac{t}{T}h\right|^{q-2}\left|\frac{h}{T}\right|^2(1-t)dt\\
      \leq & |y|^q-q|x|^{q-2}y \cdot \frac{1}{T} h + C\left(|y|^{q-2}\left|\frac{h}{T}\right|^2+\left|\frac{h}{T}\right|^q \right),
    \end{aligned}
\end{equation}

which implies
\begin{equation}\label{eta}
    \left|\dot{\eta}(s)-\frac{h}{T}\right|^q + \left|\dot{\eta}(s)+\frac{h}{T} \right|^q -2\left|\dot{\eta}\right|^q\leq C \left(\left|\frac{h}{T}\right|^2 +\left|\frac{h}{T}\right|^q \right)\leq C\left|\frac{h}{T}\right|^2
\end{equation}
where $q\geq 2$, $C=C(q, \|\dot{\eta}\|_\infty)$, and $h$ is chosen to be small enough so that $\displaystyle \left|\frac{h}{T}\right| \leq 1$.

Plugging \eqref{eta} into \eqref{diff}, we get
\begin{equation}
\begin{aligned}
     u(x+h)+u(x-h)-2u(x) &\leq C|h|^2 \int_0^T \frac{e^{-s}}{T^2}ds +C |h|^2 \int_0^Te^{-s} \left( 1-\frac{s}{T}\right)^2 ds \\
     & \leq C\frac{|h|^2}{T} \int_0^1e^{-sT}ds+C|h|^2 \int_0^T e^{-s}ds\\
     &\leq C\left(1+\frac{1}{T}\right)|h|^2\\
     &\leq C\left(1 + \frac{1}{d(x)}\right)|h|^2\\
     &\leq \frac{C}{d(x)}|h|^2
\end{aligned}
\end{equation}
since $\displaystyle T \geq \frac{d(x)}{C_0}$.
\end{itemize}
\end{proof}

%\textcolor{red}{(You only proved for $h$ small enough. What about larger $h$ such that $x+h \in \Omega$? See our definition of semiconcavity and the proof of Theorem $A.1$, where we need the inequality to hold $\forall h \in \mathbb{R}^n$ s.t. $x+h, x-h \in \Omega$. ) So in the proof of the bound for $\Delta$ in the viscosity sense, you need to show if $u-\varphi$ has min then $D^2\varphi\leq c$, thus it is enough to consider small $h$ to take the limit in the derivative only.}

\stoptocwriting
\section*{Acknowledgement}
The authors would like to express their appreciation to Hung V. Tran for his invaluable guidance. The authors would also like to thank Dohyun Kwon for useful discussions.
\resumetocwriting

% \bibliography{zzzzlibrary}{}
% %\bibliographystyle{ieeetr}
% \bibliographystyle{acm}

\end{document}